\documentclass[11pt]{amsart}

\usepackage{amsmath, amsthm, amssymb, amsfonts, amscd}

\reversemarginpar

\def\CC{{\mathbb C}}

\def\KK{{\mathbb K}}

\def\QQ{{\mathbb Q}}
\def\PP{{\mathbb P}}
\def\QQ{{\mathbb Q}}
\def\RR{{\mathbb R}}
\def\TT{{\mathbb T}}

\def\Asf{{\mathsf A}}
\def\Bsf{{\mathsf B}}

\def\Psf{{\mathsf P}}

\def\0{{\mathbf 0}}
\def\1{{\mathbf 1}}

\def\Lcal{{\mathcal L}}

\def\Ocal{{\mathcal O}}

\def\Kbar{{\bar K}}

\def\div{\mathrm{div}}

\def\Gal{\mathrm{Gal}}

\def\Spec{\mathrm{Spec}}

\def\PrePer{\mathrm{PrePer}}

\def\Res{\mathrm{Res}}

\def\supp{\mathrm{supp}}
\def\st{\mathrm{st}}

\theoremstyle{plain}

\newtheorem{thm}{Theorem}
\newtheorem{cor}[thm]{Corollary}
\newtheorem{prop}[thm]{Proposition}
\newtheorem{lem}[thm]{Lemma}
\theoremstyle{remark}
\newtheorem{ex}{Example}

\begin{document}

\title[A Dynamical Pairing Between Two Rational Maps]{A Dynamical Pairing Between Two Rational Maps}
\author{Clayton Petsche, Lucien Szpiro, and Thomas J. Tucker}

\address{Clayton Petsche; Department of Mathematics and Statistics; Hunter College; 695 Park Avenue; New York, NY 10065 U.S.A.}
\email{cpetsche@hunter.cuny.edu}

\address{Lucien Szpiro; Ph.D. Program in Mathematics; CUNY Graduate Center; 365 Fifth Avenue; New York, NY 10016-4309 U.S.A.}
\email{lszpiro@gc.cuny.edu}

\address{Thomas J. Tucker; Department of Mathematics; Hylan Building; University of Rochester; Rochester, NY 14627 U.S.A.}
\email{ttucker@math.rochester.edu}

\date{November 8, 2009}

\thanks{The first author is partially supported by NSF Grant DMS-0901147, the second author by NSF Grants DMS-0739346 and DMS-0854746, and the third author by NSF Grants DMS-0801072 and DMS-0854839.}



\begin{abstract}
Given two rational maps $\varphi$ and $\psi$ on $\PP^1$ of degree at least two, we study a symmetric, nonnegative-real-valued pairing $\langle\varphi,\psi\rangle$ which is closely related to the canonical height functions $h_\varphi$ and $h_\psi$ associated to these maps.  Our main results show a strong connection between the value of $\langle\varphi,\psi\rangle$ and the canonical heights of points which are small with respect to at least one of the two maps $\varphi$ and $\psi$.  Several necessary and sufficient conditions are given for the vanishing of $\langle\varphi,\psi\rangle$.  We give an explicit upper bound on the difference between the canonical height $h_\psi$ and the standard height $h_\st$ in terms of $\langle\sigma,\psi\rangle$, where $\sigma(x)=x^2$ denotes the squaring map.  The pairing $\langle\sigma,\psi\rangle$ is computed or approximated for several families of rational maps $\psi$.
\end{abstract}

\maketitle


\section{Introduction}\label{Introduction}

\subsection{Main results}  Let $K$ be a number field with algebraic closure $\Kbar$.  Given a rational map $\varphi:\PP^1\to\PP^1$ of degree at least two defined over $K$, and an integer $n\geq1$, denote by $\varphi^n=\varphi\circ\dots\circ\varphi$ the $n$-th iterate of $\varphi$.  In the study of the dynamical system defined by the action of the family $\{\varphi^n\}_{n=1}^{\infty}$ of all iterates of $\varphi$ on $\PP^1(\Kbar)$, a fundamental tool is the Call-Silverman canonical height function $h_\varphi:\PP^1(\Kbar)\to\RR$ associated to $\varphi$.  A basic property of this function is that $h_\varphi(x)\geq0$ for all $x\in\PP^1(\Kbar)$, with $h_\varphi(x)=0$ if and only if $x$ is preperiodic with respect to $\varphi$.  (A point $x\in\PP^1(\Kbar)$ is said to be preperiodic with respect to $\varphi$ if its forward orbit $\{\varphi^n(x)\mid n\geq1\}$ is a finite set.)

Now consider two rational maps $\varphi:\PP^1\to\PP^1$ and $\psi:\PP^1\to\PP^1$ defined over $K$, each of degree at least two.  In this paper we will study a symmetric, nonnegative-real-valued pairing $\langle\varphi,\psi\rangle$ which is closely related to the canonical height functions $h_\varphi$ and $h_\psi$.  We call $\langle\varphi,\psi\rangle$ the Arakelov-Zhang pairing, for reasons which we will explain below.  Our first result characterizes the pairing $\langle\varphi,\psi\rangle$, and illustrates its close relationship with points of small canonical height with respect to at least one of the two maps $\varphi$ and $\psi$.

\begin{thm}\label{IntroMainThm}
Let $\{x_n\}$ be a sequence of distinct points in $\PP^1(\Kbar)$ such that $h_\psi(x_n)\to0$.  Then $h_\varphi(x_n)\to\langle\varphi,\psi\rangle$.
\end{thm}

An example of a sequence $\{x_n\}$ of points in $\PP^1(\Kbar)$ with $h_\psi(x_n)\to0$ is any sequence of $\psi$-periodic points, since all such points have $h_\psi$-height zero.  In particular, averaging over all $\psi$-periodic points of a given period leads to our next result, which is similar in spirit to Theorem~\ref{IntroMainThm}, and which can be viewed as an explicit formula for $\langle\varphi,\psi\rangle$.  Define
\begin{equation}\label{SzpiroTuckerPairing}
\Sigma(\varphi,\psi)=\lim_{n\to+\infty}\frac{1}{\deg(\psi)^n+1}\sum_{\psi^n(x)=x}h_\varphi(x),
\end{equation}
where the sum is taken over all (counting according to multiplicity) $\deg(\psi)^n+1$ points $x$ in $\PP^1(\Kbar)$ satisfying $\psi^n(x)=x$, granting for now that this limit exists.  Thus $\Sigma(\varphi,\psi)$ averages the $\varphi$-canonical height of the $\psi$-periodic points of period $n$, and takes the limit as $n\to+\infty$.

\begin{thm}\label{IntroThm3}
The limit in $(\ref{SzpiroTuckerPairing})$ exists and $\Sigma(\varphi,\psi)=\langle\varphi,\psi\rangle$.
\end{thm}

In particular, it follows from the symmetry of the Arakelov-Zhang pairing that $\Sigma(\varphi,\psi)=\Sigma(\psi,\varphi)$, which is far from obvious from the definition $(\ref{SzpiroTuckerPairing})$.

Using these results we can give several necessary and sufficient conditions for the vanishing of $\langle\varphi,\psi\rangle$.  Here we denote by $\PrePer(\varphi)$ the set of all preperiodic points in $\PP^1(\Kbar)$ with respect to $\varphi$; the assumption that $\varphi$ has degree at least two ensures that $\PrePer(\varphi)$ is always an infinite set.

\begin{thm}\label{IntroThm2}
The following five conditions are equivalent:
\begin{quote}
{\bf (a)}  $\langle\varphi,\psi\rangle=0$; \\
{\bf (b)}  $h_\varphi=h_\psi$; \\
{\bf (c)}  $\PrePer(\varphi)=\PrePer(\psi)$; \\
{\bf (d)}  $\PrePer(\varphi)\cap\PrePer(\psi)$ is infinite; \\
{\bf (e)}  $\liminf_{x\in\PP^1(\Kbar)}(h_\varphi(x)+h_\psi(x))=0$.
\end{quote}
\end{thm}

Theorem~\ref{IntroThm2} states that $\langle\varphi,\psi\rangle$ vanishes precisely when $\varphi$ and $\psi$ are in some sense dynamically equivalent, and Theorems~\ref{IntroMainThm} and \ref{IntroThm3} govern the difference $h_\varphi(x)-h_\psi(x)$ between the canonical heights of points $x$ which have small canonical height with respect to at least one of the two maps $\varphi$ and $\psi$.  These results suggest that the pairing $\langle\varphi,\psi\rangle$ should be thought of as a measure of the dynamical distance between the two rational maps $\varphi$ and $\psi$.

To illustrate this idea, let $x=(x:1)$ be the usual affine coordinate on $\PP^1$, where $\infty=(1:0)$, and let $\sigma:\PP^1\to\PP^1$ be the map defined by $\sigma(x)=x^2$.  In this case the canonical height $h_{\sigma}$ is the same as the usual standard (or na\"ive) height on $\PP^1(\Kbar)$, which we denote by $h_\st$.  It then follows from Theorem~\ref{IntroThm2} that if $\psi:\PP^1\to\PP^1$ is an arbitrary map of degree at least two, then the canonical height $h_\psi$ is equal to the standard height $h_\st$ if and only if $\langle\sigma,\psi\rangle=0$.  We may therefore view $\langle\sigma,\psi\rangle$ as a natural measure of the dynamical complexity of the rational map $\psi$.  In $\S$~\ref{HeightDiffSection} we will prove the following explicit upper bound on the difference between the canonical height $h_\psi$ and the standard height $h_\st$ in terms of the pairing $\langle\sigma,\psi\rangle$.

\begin{thm}\label{HeightDiffPropIntro}
Let $\sigma:\PP^1\to\PP^1$ be the map defined by $\sigma(x)=x^2$, and let $\psi:\PP^1\to\PP^1$ be an arbitrary map of degree $d\geq2$ defined over a number field $K$.  Then
\begin{equation*}
h_\psi(x)-h_\st(x)\leq \langle\sigma,\psi\rangle + h_\psi(\infty) +\log 2
\end{equation*}
for all $x\in\PP^1(\Kbar)$.
\end{thm}

\subsection{Examples}  In $\S$~\ref{ExampleSection} we will give explicit upper and lower bounds on $\langle\sigma,\psi\rangle$ for the following families of rational maps $\psi$:
\begin{itemize}
	\item $\psi$ is the map $\sigma_\alpha(x)=\gamma_\alpha^{-1}\circ\sigma\circ\gamma_\alpha(x)=\alpha-(\alpha-x)^2$ defined by conjugating the squaring map $\sigma(x)=x^2$ by the automorphism $\gamma_\alpha(x)=\alpha-x$ of $\PP^1$, where $\alpha\in K$ is arbitrary.  In this case we will show that $\langle\sigma,\sigma_\alpha\rangle=h_\st(\alpha)+O(1)$.  In certain cases we can calculate the value of $\langle\sigma,\sigma_\alpha\rangle$ explicitly, showing that our inequalities are sharp.
	\item $\psi$ is the quadratic polynomial $\psi_c(x)=x^2+c$, where $c\in K$ is arbitrary.  In this case we will show that $\langle\sigma,\psi_c\rangle=\frac{1}{2}h_\st(c)+O(1)$.
	\item $\psi$ is the Latt\`es map $\psi_E(x)=(x^2+ab)^2/4x(x-a)(x+b)$ associated to the doubling map on the elliptic curve $E$ given by the Weierstrass equation $y^2=x(x-a)(x+b)$, where $a$ and $b$ are positive integers.  In this case we will show that $\langle\sigma,\psi_E\rangle=\log\sqrt{ab}+O(1)$.
\end{itemize}

\subsection{Summary of methods}  Our definition of the Arakelov-Zhang pairing $\langle\varphi,\psi\rangle$ relies on local analytic machinery.  We define $\langle\varphi,\psi\rangle$ as a sum of local terms of the form $-\int\lambda_{\varphi,v}\Delta\lambda_{\psi,v}$, where $\lambda_{\varphi,v}$ and $\lambda_{\psi,v}$ are canonical local height functions associated to $\varphi$ and $\psi$, and where $\Delta$ is a suitable $v$-adic Laplacian operator.  As shown by several authors, among them Baker-Rumely \cite{BakerRumelyBook}, Favre-Rivera-Letelier \cite{FavreRiveraLetelier}, and Thuillier \cite{Thuillier}, the natural space on which these analytic objects are defined is not the ordinary projective line $\PP^1(\CC_v)$ at the place $v$, but rather the Berkovich projective line $\Psf^1_v$.  (Briefly, $\Psf^1_v$ is a compactification of $\PP^1(\CC_v)$ which, at the non-archimedean places, has a richer analytic structure than $\PP^1(\CC_v)$; we will review the necessary facts about this space in $\S$~\ref{BerkovichSect}.)

Our proof of Theorem~\ref{IntroMainThm} relies on the equidistribution theorem for dynamically small points on $\PP^1$, which is due independently to Baker-Rumely \cite{BakerRumely}, Chambert-Loir \cite{ChambertLoir}, and Favre-Rivera-Letelier \cite{FavreRiveraLetelier}, and which is a dynamical analogue of equidistributon results on abelian varieties and algebraic tori due to Szpiro-Ullmo-Zhang \cite{SzpiroUllmoZhang} and Bilu \cite{Bilu}, respectively.  This result states that if $\{x_n\}$ is a sequence of distinct points in $\PP^1(\Kbar)$ such that $h_\psi(x_n)\to0$, then the sets of $\Gal(\Kbar/K)$-conjugates of the terms $x_n$ equidistribute with respect to the canonical measure $\mu_{\psi,v}$ on each local analytic space $\Psf_v^1$.  It follows that the heights $h_\varphi(x_n)$ tends toward an expression involving integrals of local height functions with respect to $\varphi$ against canonical measures with respect to $\psi$; this expression is precisely equal to the Arakelov-Zhang pairing $\langle\varphi,\psi\rangle$.  The proof of Theorem~\ref{IntroThm3} is similar to the proof of Theorem~\ref{IntroMainThm}, and Theorem~\ref{IntroThm2} is proved using Theorem~\ref{IntroMainThm} along with basic properties of canonical local height functions and canonical measures.  Theorem~\ref{HeightDiffPropIntro} is proved using Theorem~\ref{IntroThm3} along with the generalized Mahler formula of Szpiro-Tucker \cite{SzpiroTucker}.

\subsection{Related work by other researchers}  The pairing $\langle\varphi,\psi\rangle$ is equivalent to the arithmetic intersection product, originally defined by Zhang \cite{Zhang} following earlier work of Arakelov \cite{Arakelov}, Deligne \cite{Deligne}, Faltings \cite{Faltings}, and Bost-Gillet-Soul\'e \cite{BostGilletSoule}, between the canonical adelic metrized line bundles on $\PP^1$ associated to $\varphi$ and $\psi$.  Denote by $\Lcal_\varphi$ the canonical adelic metrized line bundle $(\Ocal(1),\|\cdot\|_{\varphi})$ on $\PP^1$ associated to $\varphi$, and define $\Lcal_\psi$ likewise for $\psi$ (we will review the definitions of these objects in $\S$~\ref{CanMetricSect}).  Then
\begin{equation}\label{ZhangPairing}
\langle\varphi,\psi\rangle = c_1(\Lcal_\varphi)c_1(\Lcal_\psi),
\end{equation}
where the left-hand-side denotes the Arakelov-Zhang pairing, as we define it analytically in $\S$~\ref{GlobalAZPairingSect}, and the right-hand-side denotes Zhang's \cite{Zhang} arithmetic intersection product between (the first Chern classes of) the adelic metrized line bundles $\Lcal_\varphi$ and $\Lcal_\psi$.

To briefly summarize Zhang's approach, he first expresses the adelic metrized line bundles $\Lcal_\varphi$ and $\Lcal_\psi$ as uniform limits of sequences $\{\Lcal_{\varphi,k}\}_{k=0}^{\infty}$ and $\{\Lcal_{\psi,k}\}_{k=0}^{\infty}$ of adelic metrized line bundles arising from arithmetic models of $\PP^1$ over $\Spec\,\Ocal_K$.  He then defines the products $c_1(\Lcal_{\varphi,k})c_1(\Lcal_{\psi,k})$ via the traditional Arakelov-theoretic combination of intersection theory at the non-archimedean places, along with harmonic analysis and Green's functions at the archimedean places.  Finally he defines $c_1(\Lcal_\varphi)c_1(\Lcal_\psi)$ as the limit of the sequence $\{c_1(\Lcal_{\varphi,k})c_1(\Lcal_{\varphi,k})\}_{k=0}^{\infty}$.  More recently, Chambert-Loir \cite{ChambertLoir} has put these intersection products onto a more analytic footing by expressing them as sums of local integrals against certain measures on the Berkovich projective line $\Psf_v^1$.  We should point that both Zhang's and Chambert-Loir's work holds for a much more general class of adelic metrized line bundles on varieties of arbitrary dimension; in this paper we treat only canonical adelic metrized line bundles $\Lcal_\varphi$ arising from dynamical systems on the projective line $\PP^1$.  We do not require the identity $(\ref{ZhangPairing})$ in this paper, and so we will not actually work out the details of the proof.  The equality between the two pairings is discussed in \cite{ChambertLoir} $\S$~2.9.

Part of Theorem~\ref{IntroThm2} is equivalent via $(\ref{ZhangPairing})$ to known results in the literature.  In particular, Zhang's successive minima theorem (\cite{Zhang} Thm.~1.10) implies the equivalence of {\bf (a)} and {\bf (e)} in Theorem~\ref{IntroThm2}, and Mimar \cite{Mimar} has shown the equivalence of conditions {\bf (c)}, {\bf (d)}, and {\bf (e)}.  The main novelty of our result is that {\bf (b)} follows from the other conditions, and that when armed with the equidistribution theorem, it is very simple to prove the equivalence of {\bf (a)}, {\bf (c)}, {\bf (d)}, and {\bf (e)}.

Kawaguchi-Silverman \cite{KawaguchiSilverman} have studied the problem of what can be deduced about two morphisms $\varphi,\psi:\PP^N\to\PP^N$ (of degree at least two) under the hypothesis that $h_\varphi=h_\psi$.  On $\PP^1$ they give a complete classification of such pairs under the additional assumptions either that both $\varphi$ and $\psi$ are polynomials, or that at least one of $\varphi$ and $\psi$ is a Latt\`es map associated to an elliptic curve.

Just before submitting this article for publication, we learned that Theorem~\ref{IntroMainThm} is equivalent to a special case of a result of Chambert-Loir and Thuillier in their recent work on equidistribution of small points on varieties with semi-positive adelic metrics; see \cite{ChambertLoirThuillier} $\S$~6.

Baker-DeMarco have recently released a preprint \cite{BakerDeMarco} which, among other things, shows that if $\varphi,\psi:\PP^1\to\PP^1$ are two rational maps (of degree at least two) defined over $\CC$, then $\PrePer(\varphi)=\PrePer(\psi)$ if and only if $\PrePer(\varphi)\cap\PrePer(\psi)$ is infinite.  This generalizes Mimar's result from maps defined over $\overline{\QQ}$ to those defined over $\CC$.  Yuan-Zhang have announced a generalization of this result to arbitrary polarized algebraic dynamical systems over $\CC$.


\section{Algebraic Preliminaries}

\subsection{Homogeneous lifts and polarizations}\label{PolarizationSect}  Let $k$ be an arbitrary field, and fix homogeneous coordinates $(x_0:x_1)$ on the projective line $\PP^1$ over $k$.  Let $\varphi:\PP^1\to\PP^1$ be a rational map of degree $d\geq2$.  A homogeneous lift of $\varphi$ is a map $\Phi=(\Phi_0,\Phi_1):k^2\to k^2$ satisfying
\begin{equation*}
\varphi(x_0:x_1)=(\Phi_{0}(x_0,x_1):\Phi_{1}(x_0,x_1))
\end{equation*}
for all $(x_0:x_1)\in\PP^1(k)$.

A polarization of $\varphi$ is a $k$-isomorphism $\epsilon:\Ocal(d)\stackrel{\sim}{\to}\varphi^*\Ocal(1)$ of sheaves.  The choice of a homogeneous lift $\Phi$ and the choice of a polarization $\epsilon$ are equivalent in the following sense.  Viewing the coordinates $x_0$ and $x_1$ on $\PP^1$ as sections $x_j\in\Gamma(\PP^1,\Ocal(1))$, define sections $\Phi_{\epsilon,0},\Phi_{\epsilon,1}\in\Gamma(\PP^1,\Ocal(d))$ by $\Phi_{\epsilon,j}=\epsilon^*\varphi^*x_j$; thus we view $\Phi_{\epsilon,0}(x_0,x_1)$ and $\Phi_{\epsilon,1}(x_0,x_1)$ as homogeneous forms in $k[x_0,x_1]$ of degree $d$, and the resulting map $\Phi_\epsilon=(\Phi_{\epsilon,0},\Phi_{\epsilon,1}):k^2\to k^2$ is a homogeneous lift of $\varphi$.  Conversely, any homogeneous lift $\Phi:k^2\to k^2$ of $\varphi$ determines a polarization $\epsilon:\Ocal(d)\stackrel{\sim}{\to}\varphi^*\Ocal(1)$ such that $\Phi=\Phi_\epsilon$.  Both $\epsilon$ and $\Phi$ are unique up to multiplication by a nonzero scalar in $k$.

\subsection{Resultants}\label{ResultantSect}  Let $\Phi:k^2\to k^2$ be a map defined by a pair $(\Phi_0(x_0,x_1),\Phi_1(x_0,x_1))$ of homogeneous forms of degree $d$ in $k[x_0,x_1]$.  The resultant of $\Phi$ is an element $\Res(\Phi)$ of $k$ which is defined by a certain integer polynomial in the coefficients of $\Phi$.  The most important property of the resultant is that $\Res(\Phi)=0$ if and only if $\Phi(x_0,x_1)=0$ for some nonzero $(x_0,x_1)\in \bar{k}$.  Thus $\Res(\Phi)=0$ if and only if $\Phi$ is a homogeneous lift of a rational map $\varphi:\PP^1\to\PP^1$.  For the definition and basic theory of the resultant see \cite{vanderWaerden} $\S$~82.


\section{Local Considerations}\label{LocalSection}

\subsection{Notation}  Throughout this section the pair $(\KK,|\cdot|)$ denotes either the complex field $\KK=\CC$ with its usual absolute value $|\cdot|$, or an arbitrary algebraically closed field $\KK$ which is complete with respect to a non-trivial, non-archimedean absolute value $|\cdot|$.  In the latter case we let $\KK^\circ= \{a\in\KK\mid|a|\leq1\}$ denote the valuation ring of $\KK$, with maximal ideal $\KK^{\circ\circ}=\{a\in\KK\mid|a|<1\}$ and residue field $\tilde{\KK}=\KK^\circ/\KK^{\circ\circ}$.  Fix once and for all homogeneous coordinates $(x_0:x_1)$ on the projective line $\PP^1$ over $\KK$.

\subsection{The Berkovich affine and projective lines}\label{BerkovichSect}  In this section we will review the definitions of the Berkovich affine and projective lines.  For more details on these objects and for the proofs of the claims we make here, see \cite{BakerRumelyBook} or \cite{Berkovich}.

The Berkovich affine line $\Asf^1$ is defined to be the set of multiplicative seminorms on the polynomial ring $\KK[T]$ in one variable.  $\Asf^1$ is a locally compact, Hausdorff topological space with respect to the weakest topology under which all real-valued functions of the form $x\mapsto[f(T)]_x$ are continuous.  Here and throughout this paper we use $[\cdot]_x$ to denote the seminorm corresponding to the point $x\in\Asf^1$.  Observe that, given an element $a\in\KK$, we have the evaluation seminorm $[f(T)]_a=|f(a)|$.  The map $a\mapsto[\cdot]_a$ defines a dense embedding $\KK\hookrightarrow\Asf^1$; it is customary to regard this as an inclusion map, and thus one identifies each point $a\in\KK$ with its corresponding seminorm $[\cdot]_a$ in $\Asf^1$.  The Berkovich projective line $\Psf^1$ over $\KK$ is defined to be the one-point-compactification $\Psf^1=\Asf^1\cup\{\infty\}$ of $\Asf^1$.  The dense inclusion map $\KK\hookrightarrow\Asf^1$ extends to a dense inclusion map $\PP(\KK)\hookrightarrow\Psf^1$ defined by $(a:1)\mapsto[\cdot]_a$ and $(1:0)\mapsto\infty$.  Since $\Asf^1$ is locally compact, $\Psf^1$ is compact.

In the archimedean case ($\KK=\CC$), it turns out that the evaluation seminorms are the {\em only} seminorms.  Thus $\Asf^1=\CC$ and $\Psf^1=\PP^1(\CC)$.  In particular, $\Psf^1$ is a compact Riemann surface of genus zero.

When $\KK$ is non-archimedean, however, the inclusion $\PP(\KK)\hookrightarrow\Psf^1$ is far from surjective.  For example, each closed disc $B(a,r)=\{z\in\KK \mid |z-a|\leq r \}$ ($a\in\KK$, $r\in|\KK|$) defines a point $\zeta_{a,r}$ in $\Asf^1$ corresponding to the sup norm $[f(T)]_{\zeta_{a,r}}=\sup_{z\in B(a,r)}|f(z)|$.  Note that under this notation, each $a\in\KK$ can also be written as the point $\zeta_{a,0}\in\Asf^1$ corresponding to a disc of radius zero.

Finally, we observe that since $\KK$ is dense in $\Asf^1$, the continuous function $\KK\to\RR$ defined by $a\mapsto|a|$ has a unique continuous extension $|\cdot|:\Asf^1\to\RR$ defined by $x\mapsto[T]_x$.  By a slight abuse of notation we will still use the notation $|\cdot|$ to refer to this extended function.  Declaring $|\infty|=+\infty$, we obtain a continuous function $|\cdot|:\Psf^1\to\RR\cup\{+\infty\}$.

\subsection{The measure-valued Laplacian}  The measure-valued Laplacian $\Delta$ on the Berkovich projective line $\Psf^1$ is an operator which assigns to a continuous function $f:\Psf^1\to\RR\cup\{\pm\infty\}$ (satisfying sufficient regularity conditions) a signed Borel measure $\Delta f$ on $\Psf^1$.

When $\KK=\CC$, $\Delta$ is the usual normalized $-dd^c$ operator on $\Psf^1=\PP^1(\CC)$ as a compact Riemann surface.  For example, assume that $f:\Psf^1\to\RR$ is twice continuously real-differentiable, identify $\Psf^1=\PP^1(\CC)=\CC\cup\{\infty\}$ by the affine coordinate $z=(z:1)\in \CC$, where $\infty=(1:0)$, and let $z=x+iy$ for real variables $x$ and $y$.  Then
\begin{equation*}
\Delta f(z)= -\frac{1}{2\pi}\bigg(\frac{\partial^2}{\partial x^2}+\frac{\partial^2}{\partial y^2}\bigg)f(z)dx\,dy.
\end{equation*}

It has been shown by several authors, including Baker-Rumely \cite{BakerRumelyBook}, Favre-Rivera-Letelier \cite{FavreRiveraLetelier}, and Thuillier \cite{Thuillier}, that the non-archimedean Berkovich projective line $\Psf^1$ carries an analytic structure, and in particular a Laplacian, which is very similar to its archimedean counterpart.  In this paper we follow the approach of Baker-Rumely \cite{BakerRumelyBook}.  To summarize their construction, viewing $\Psf^1$ as an inverse limit of finitely branched metrized graphs, they first define a signed Borel measure $\Delta f$ for a certain class of functions $f:\Psf^1\to\RR$ which are locally constant outside of some finitely-branched subgraph of $\Psf^1$.  Passing to the limit, they introduce a space of functions $f:\Psf^1\to\RR\cup\{\pm\infty\}$, which is in effect the largest class for which $\Delta f$ exists as a signed Borel measure on $\Psf^1$.

The following proposition summarizes the two basic properties of the measure-valued Laplacian which we will need in this paper: the self-adjoint property and the criterion for vanishing Laplacian.  It is stated to hold in both the archimedean and non-archimedean cases.

\begin{prop}
The measure-valued Laplacian $\Delta$ on $\Psf^1$ satisfies the following properties:
\begin{quote}
{\bf (a)} $\int fd(\Delta g)=\int gd(\Delta f)$ whenever $f$ is $\Delta g$-integrable and $g$ is $\Delta f$-integrable. \\
{\bf (b)} $\Delta f=0$ if and only if $f$ is constant.
\end{quote}
\end{prop}

In the archimedean case these facts are standard.  In the non-archimedean case see \cite{BakerRumelyBook} Prop.~5.20 and Prop.~5.28.

Note that the total mass of $\Delta f$ is always zero; that is $\int 1d(\Delta f)=0$.  This follows from the self-adjoint property along with the fact that $\Delta 1=0$.

\subsection{The standard metric}\label{StandardMetricSect}  Given a line bundle $\Lcal$ on $\PP^1$ over $\KK$, recall that a metric $\|\cdot\|$ on $\Lcal$ is a nonnegative-real-valued function on $\Lcal$ whose restriction $\|\cdot\|_x$ to each fiber $\Lcal_x$ is a norm on $\Lcal_x$ as a $\KK$-vector space.  The standard metric $\|\cdot\|_\st$ on the line bundle $\Ocal(1)$ over $\KK$ is characterized by the identity
\begin{equation}\label{StandardMetric}
\|s(x)\|_\st=|s(x_0,x_1)|/\max\{|x_0|,|x_1|\}
\end{equation}
for each section $s\in\Gamma(\PP^1,\Ocal(1))$ given as a linear form $s(x_0,x_1)\in \KK[x_0,x_1]$.

\subsection{Canonical metrics}\label{CanMetricSect}  Let $\varphi:\PP^1\to\PP^1$ be a rational map of degree $d\geq2$ defined over $\KK$, and let $\epsilon:\Ocal(d)\stackrel{\sim}{\to}\varphi^*\Ocal(1)$ be a polarization.  The canonical metric on $\Ocal(1)$ associated to the pair $(\varphi,\epsilon)$, introduced by Zhang \cite{Zhang} (see also \cite{BombieriGubler} $\S$9.5.3), is defined as the (uniform) limit of the sequence $\{\|\cdot\|_{\varphi,\epsilon,k}\}_{k=0}^{\infty}$ of metrics on $\Ocal(1)$ defined inductively by $\|\cdot\|_{\varphi,\epsilon,0}=\|\cdot\|_{\st}$ and $\epsilon^*\varphi^*\|\cdot\|_{\varphi,\epsilon,k}=\|\cdot\|_{\varphi,\epsilon,k+1}^{\otimes d}$.  Zhang \cite{Zhang} showed that such a metric $\|\cdot\|_{\varphi,\epsilon}$ on $\Ocal(1)$ exists, and that it is the unique bounded, continuous metric on $\Ocal(1)$ satifying
\begin{equation}\label{CanMetricIdentity}
\epsilon^*\varphi^*\|\cdot\|_{\varphi,\epsilon}=\|\cdot\|_{\varphi,\epsilon}^{\otimes d}.
\end{equation}

To make this more explicit, let $s\in\Gamma(\PP^1,\Ocal(1))$ be a section defined over $\KK$, and let $u=\epsilon^*\varphi^*s\in\Gamma(\PP^1,\Ocal(d))$.  Factoring $u=\otimes_{j=1}^{d}s_j$ for sections $s_j\in\Gamma(\PP^1,\Ocal(1))$, it follows from the identity $(\ref{CanMetricIdentity})$ that
\begin{equation}\label{CanMetricIdentityExp}
\|s(\varphi(x))\|_{\varphi,\epsilon}=\prod_{j=1}^{d}\|s_j(x)\|_{\varphi,\epsilon}
\end{equation}
for all $x\in \PP^1(\KK)$.  The boundedness and continuity conditions mean, in effect, that for each section $s\in\Gamma(\PP^1,\Ocal(1))$ defined over $\KK$ the function $x\mapsto \log(\|s(x)\|_{\varphi,\epsilon}/\|s(x)\|_\st)$ is bounded and continuous on $\PP^1(\KK)$.

The dependence of $\|\cdot\|_{\varphi,\epsilon}$ on the polarization $\epsilon$ can be made explicit as follows.  If $\epsilon$ is replaced by another polarization $a\epsilon$ for $a\in \KK^\times$, then $\|\cdot\|_{\varphi,a\epsilon}=|a|^{1/(d-1)}\|\cdot\|_{\varphi,\epsilon}$; this follows from $(\ref{CanMetricIdentity})$ and the uniqueness of the canonical metric.

\subsection{The standard measure}\label{StandardMeasureSect}  Given a section $s\in\Gamma(\PP^1,\Ocal(1))$, the function $\PP^1(\KK)\to\RR\cup\{+\infty\}$ defined by $x\mapsto-\log\|s(x)\|_{\st}$ extends uniquely to a continuous function $\Psf^1\to\RR\cup\{+\infty\}$; this follows from $(\ref{StandardMetric})$ and the last paragraph of $\S$~\ref{BerkovichSect}.  The standard measure $\mu_\st$ on $\Psf^1$ is characterized by the identity
\begin{equation}\label{StandardMeasureDef}
\Delta\{-\log\|s(x)\|_{\st}\}=\delta_{\div(s)}(x)-\mu_\st(x),
\end{equation}
for any section $s\in\Gamma(\PP^1,\Ocal(1))$, where $\delta_{\div(s)}$ denotes the Dirac measure supported at the point $\div(s)$.

In fact, $\mu_\st$ can be described more explicitly as follows.  When $\KK=\CC$, we identify $\Psf^1=\PP^1(\CC)=\CC\cup\{\infty\}$ as usual by the affine coordinate $z=(z:1)\in \CC$, where $\infty=(1:0)$.  Then $\mu_\st$ is supported on the unit circle $\TT=\{z\in\CC\mid |z|=1\}$, where it is equal to the Haar measure on $\TT$ normalized to have total mass $1$.  When $\KK$ is non-archimedean, $\mu_\st$ is equal to the Dirac measure $\delta_{\zeta_{0,1}}$ supported at point $\zeta_{0,1}\in\Psf^1$ corresponding to the sup-norm on the unit disc of $\KK$, as described in $\S$~\ref{BerkovichSect}.

\subsection{Canonical measures}\label{CanonicalMeasureSect}  Let $\varphi:\PP^1\to\PP^1$ be a rational map of degree $d\geq2$ defined over $\KK$.  The canonical measure $\mu_\varphi$ on $\Psf^1$ is an important dynamical invariant associated to $\varphi$, which plays a fundamental role in several dynamical equidistribution theorems.  In the archimedean case it was introduced by Brolin \cite{Brolin} for polynomial maps $\varphi$, and independently by Ljubich \cite{Ljubich} and Freire-Lopes-Ma\~n\'e \cite{FreireLopesMane} for arbitrary $\varphi$.  In the non-archimedean case the measure was defined independently and equivalently by several authors, among them Baker-Rumely \cite{BakerRumely}, Chambert-Loir \cite{ChambertLoir}, and Favre-Rivera-Letelier \cite{FavreRiveraLetelier}.

The measure $\mu_\varphi$ is defined as the weak limit of the sequence $\{\mu_{\varphi,k}\}_{k=0}^{\infty}$ of measures on $\Psf^1$ defined inductively by $\mu_{\varphi,0}=\mu_\st$ and $\mu_{\varphi,k+1}=\frac{1}{d}\varphi^*\mu_{\varphi,k}$.  Equivalently, each measure $\mu_{\varphi,k}$ is characterized by the identity
\begin{equation*}
\Delta\{-\log\|s(x)\|_{\varphi,\epsilon,k}\}=\delta_{\div(s)}(x)-\mu_{\varphi,k}(x),
\end{equation*}
where $\epsilon:\Ocal(d)\stackrel{\sim}{\to}\varphi^*\Ocal(1)$ is any polarization, $\{\|\cdot\|_{\varphi,\epsilon,k}\}_{k=0}^{\infty}$ is the sequence of metrics defined in $\S$~\ref{CanMetricSect}, and $s\in\Gamma(\PP^1,\Ocal(1))$ is any section.  The canonical measure $\mu_\varphi$ on $\Psf^1$ is defined to be the (unique) weak limit of the sequence $\{\mu_{\varphi,k}\}_{k=0}^{\infty}$.  This measure is characterized by the identity
\begin{equation}\label{CanMeasureDef}
\Delta\{-\log\|s(x)\|_{\varphi,\epsilon}\}=\delta_{\div(s)}(x)-\mu_\varphi(x),
\end{equation}
where $\|\cdot\|_{\varphi,\epsilon}$ is the canonical metric on $\Ocal(1)$ with respect to the pair $(\varphi,\epsilon)$.  Moreover, $\mu_\varphi$ satisfies the invariance property $\varphi^*\mu_\varphi=d\cdot\mu_\varphi$; see \cite{BakerRumely} Thm. 3.36.  In effect, this means that
\begin{equation}\label{CanMeasureInvariant}
\int f(x)d\mu_\varphi(x) = \int f(\varphi(x))d\mu_\varphi(x)
\end{equation}
for any extended-real-valued function $f$ on $\Psf^1$ such that both $x\mapsto f(x)$ and $x\mapsto f(\varphi(x))$ are $\mu_\varphi$-integrable.

\subsection{The local Arakelov-Zhang pairing}  Let $\varphi:\PP^1\to\PP^1$ and $\psi:\PP^1\to\PP^1$ be two rational maps defined over $\KK$, and let $s,t\in\Gamma(\PP^1,\Ocal(1))$ be two sections with $\div(s)\neq\div(t)$.  We define the local Arakelov-Zhang pairing of $\varphi$ and $\psi$, with respect to the sections $s$ and $t$, by
\begin{equation}\label{LocalAZPairing}
\begin{split}
\langle\varphi,\psi\rangle_{s,t} & = -\int\{\log\|s(x)\|_{\varphi,\epsilon_\varphi}\}d\Delta\{\log\|t(x)\|_{\psi,\epsilon_\psi}\} \\
	& = \log\|s(\div(t))\|_{\varphi,\epsilon_\varphi} - \int\log\|s(x)\|_{\varphi,\epsilon_\varphi} d\mu_{\psi}(x).
\end{split}
\end{equation}
where $\epsilon_\varphi$ and $\epsilon_\psi$ are any polarizations of $\varphi$ and $\psi$, respectively.  Note that $\langle\varphi,\psi\rangle_{s,t}$ does not depend on the choice of polarzations $\epsilon_\varphi$ and $\epsilon_\psi$.  To see this, note that the discussion at the end of $\S$~\ref{CanMetricSect} implies that if we replace $\epsilon_\varphi$ and $\epsilon_\psi$ by $\epsilon'_\varphi=a\epsilon_\varphi$ and $\epsilon'_\psi=b\epsilon_\psi$ (respectively) for $a,b\in\KK^\times$, then the functions $-\log\|s(x)\|_\varphi$ and $-\log\|t(x)\|_\psi$ are altered by additive constants.  Since $\Delta1=0$ and $\int 1 d\Delta\{\log\|t(x)\|_\psi\}=0$, the value of $(\ref{LocalAZPairing})$ remains unchanged.

We next observe that the local Arakelov-Zhang pairing satisfies the symmetry property
\begin{equation}\label{LocalSymmetry}
\langle\varphi,\psi\rangle_{s,t} = \langle\psi,\varphi\rangle_{t,s},
\end{equation}
by the self-adjoint property of the measure-valued Laplacian.

Finally, we note that the value of $\langle\varphi,\psi\rangle_{s,t}$ depends on the choice of sections $s$ and $t$ as follows.  If $t'\in\Gamma(\PP^1,\Ocal(1))$ is another section with $\div(s)\neq\div(t')$, then $(\ref{LocalAZPairing})$ implies that
\begin{equation}\label{SectionChange1}
\langle\varphi,\psi\rangle_{s,t'} = \langle\varphi,\psi\rangle_{s,t} + \log\frac{\|s(\div(t'))\|_{\varphi,\epsilon_\varphi}}{\|s(\div(t))\|_{\varphi,\epsilon_\varphi}}.
\end{equation}
The symmetry relation $(\ref{LocalSymmetry})$ along with $(\ref{SectionChange1})$ implies that if $s'\in\Gamma(\PP^1,\Ocal(1))$ is another section with $\div(s')\neq\div(t)$, then
\begin{equation}\label{SectionChange2}
\langle\varphi,\psi\rangle_{s',t} = \langle\varphi,\psi\rangle_{s,t} + \log\frac{\|t(\div(s'))\|_{\psi,\epsilon_\psi}}{\|t(\div(s))\|_{\psi,\epsilon_\psi}}.
\end{equation}

\subsection{Good reduction}  In section we assume that $\KK$ is non-archimedean.  Let $\varphi:\PP^1\to\PP^1$ be a rational map of degree $d\geq2$ defined over $\KK$, let $\epsilon:\Ocal(d)\stackrel{\sim}{\to}\varphi^*\Ocal(1)$ be a polarization, and let $\Phi_\epsilon:\KK^2\to\KK^2$ be the homogenous map associated to the pair $(\varphi,\epsilon)$, as discussed in $\S$~\ref{PolarizationSect}.  Note that $\Res(\Phi_\epsilon)\neq0$ since $\Phi_\epsilon$ is a homogeneous lift of the rational map $\varphi$.  We say the pair $(\varphi,\epsilon)$ has good reduction if the map $\Phi_\epsilon$ has coefficients in $\KK^\circ$ and $|\Res(\Phi_\epsilon)|=1$.  In this case the reduced map $\tilde{\Phi}_\epsilon:\tilde{\KK}^2\to\tilde{\KK}^2$ has nonzero resultant $\Res(\tilde{\Phi}_\epsilon)$ in the residue field $\tilde{\KK}$, and therefore defines a reduced rational map $\tilde{\varphi}:\PP^1\to\PP^1$ over $\tilde{\KK}$.

\begin{prop}\label{GoodReductionProp}
Let $\varphi:\PP^1\to\PP^1$ be a rational map of degree $d\geq2$ defined over $\KK$, let $\epsilon$ be a polarization of $\varphi$, and suppose that the pair $(\varphi,\epsilon)$ has good reduction.  Then $\|\cdot\|_{\varphi,\epsilon}=\|\cdot\|_\st$ and $\mu_{\varphi}=\mu_\st$.
\end{prop}
\begin{proof}
We will first show that
\begin{equation}\label{NormIdentity}
\max\{|\Phi_{\epsilon,0}(x_0,x_1)|,|\Phi_{\epsilon,1}(x_0,x_1)|\}=\max\{|x_0|,|x_1|\}^d
\end{equation}
for all $(x_0,x_1)\in\KK^2$.  By scaling it suffices to prove $(\ref{NormIdentity})$ when $\max\{|x_0|,|x_1|\}=1$.  In this case the left-hand-side of  $(\ref{NormIdentity})$ is at most $1$ by the ultrametric inequality, since $\Phi_\epsilon$ has coefficients in $\KK^\circ$.  If there exists some $(x_0,x_1)\in\KK^2$ with $\max\{|x_0|,|x_1|\}=1$, but with the left-hand-side of $(\ref{NormIdentity})$ strictly less than $1$, then the reduction $(\tilde{x}_0,\tilde{x}_1)$ is nonzero in $\tilde{\KK}^2$, but $\tilde{\Phi}_\epsilon(\tilde{x}_0,\tilde{x}_1)=0$ in $\tilde{\KK}^2$, contradicting the fact that $\Res(\tilde{\Phi}_\epsilon)\neq0$ in $\tilde{\KK}$.  This completes the proof of $(\ref{NormIdentity})$.

Now let $s\in\Gamma(\PP^1,\Ocal(1))$ be a section, let $u=\epsilon^*\varphi^*s\in\Gamma(\PP^1,\Ocal(d))$, and factor $u=\otimes_{j=1}^{d}s_j$ for sections $s_j\in\Gamma(\PP^1,\Ocal(1))$.  It follows from $(\ref{NormIdentity})$ and the formula $(\ref{StandardMetric})$ for the standard metric that $\|s(\varphi(x))\|_{\st}=\prod_{j=1}^{d}\|s_j(x)\|_{\st}$ for all $x\in \PP^1(\KK)$.  In other words, the standard metric $\|\cdot\|_\st$ satisfies the identity which characterizes the canonical metric $\|\cdot\|_{\varphi,\epsilon}$, whereby $\|\cdot\|_{\varphi,\epsilon}=\|\cdot\|_\st$.  It now follows from $(\ref{CanMeasureDef})$ and $(\ref{StandardMeasureDef})$ that $\mu_{\varphi}=\mu_\st$.
\end{proof}

\begin{prop}\label{GoodReductionAZProp}
Let $\varphi:\PP^1\to\PP^1$ and $\psi:\PP^1\to\PP^1$ be rational maps of degree at least two defined over $\KK$, and assume that both $\varphi$ and $\psi$ have good reduction (with respect to some choice of polarizations).  Let $s,t\in\Gamma(\PP^1,\Ocal(1))$ be sections, with $\div(s)\neq\div(t)$, given by linear forms $s(x_0,x_1)=s_0x_0+s_1x_1$ and $t(x_0,x_1)=t_0x_0+t_1x_1$ in $\KK[x_0,x_1]$.  Then
\begin{equation*}
\langle\varphi,\psi\rangle_{s,t}=\log|s_0t_1-s_1t_0|-\log\max\{|s_0|,|s_1|\}-\log\max\{|t_0|,|t_1|\}.
\end{equation*}
\end{prop}
\begin{proof}
Let us abbreviate $I(s)=\int\log\|s(x)\|_{\st} d\mu_{\st}(x)$ and $\ell(s)=\log\max\{|s_0|,|s_1|\}$.  Since $\div(t)=(t_1:-t_0)$, by $(\ref{LocalAZPairing})$ and Proposition~\ref{GoodReductionProp} we have
\begin{equation}\label{GoodReductionPairing}
\langle\varphi,\psi\rangle_{s,t} = \log\|s(\div(t))\|_{\st} - I(s) = \log|s_0t_1-s_1t_0| -\ell(t) - I(s).
\end{equation}
The symmetry property $\langle\varphi,\psi\rangle_{s,t}=\langle\psi,\varphi\rangle_{t,s}$ and $(\ref{GoodReductionPairing})$ imply that $-\ell(t) - I(s)=-\ell(s) - I(t)$, which means that $I(s)-\ell(s)=I(t)-\ell(t)=c$ for some constant $c$ which is independent of $s$ and $t$.  We are going to show that $c=0$, whereby $I(s)=\ell(s)$; together with $(\ref{GoodReductionPairing})$ this implies the desired identity.

To show $c=0$ we select the section $s(x_0,x_1)=x_0$; thus $\ell(s)=0$.  Let $\epsilon:\Ocal(d)\stackrel{\sim}{\to}\varphi^*\Ocal(1)$ be a polarization of $\varphi$ such that $(\varphi,\epsilon)$ has good reduction, and factor $\Phi_{\epsilon,0}=\epsilon^*\varphi^*s\in\Gamma(\PP^1,\Ocal(d))$ as $\Phi_{\epsilon,0}=\otimes_{j=1}^{d}s_j$ for sections $s_j\in\Gamma(\PP^1,\Ocal(1))$.  In terms of homogeneous forms in the coordinates $x_0$ and $x_1$ this means that $\Phi_{\epsilon,0}(x_0,x_1)=\prod_{j=1}^{d}s_j(x_0,x_1)$  The good reduction assumption on the pair $(\varphi,\epsilon)$ implies that the form $\Phi_{\epsilon,0}(x_0,x_1)$ is defined over $\KK^\circ$ and reduces to a nonzero form $\tilde{\Phi}_{\epsilon,0}(x_0,x_1)$ of degree $d$ over $\tilde{\KK}$.  In particular, each linear form $s_j(x_0,x_1)$ is defined over $\KK^\circ$ and reduces to a nonzero linear form $\tilde{s}_j(x_0,x_1)$ over $\tilde{\KK}$.  This means that $\ell(s_j)=0$ for each $1\leq j\leq d$.  Using the invariance property $(\ref{CanMeasureInvariant})$ along with $(\ref{CanMetricIdentityExp})$ and the fact that $\|\cdot\|_{\varphi,\epsilon}=\|\cdot\|_{\st}$, we have
\begin{equation*}
\int\log\|s(x)\|_{\st} d\mu_{\st}(x)  = \int\log\|s(\varphi(x))\|_{\st} d\mu_{\st}(x) = \sum_{j=1}^{d}\int\log\|s_j(x)\|_{\st} d\mu_{\st}(x),
\end{equation*}
which means that $I(s)=\sum_{j=1}^{d}I(s_j)$.  Since $\ell(s)=\ell(s_j)=0$ for each $1\leq j\leq d$, we conclude that $c=I(s)-\ell(s)=\sum_{j=1}^{d}(I(s_j)-\ell(s_j))=dc$, whereby $c=0$ as desired.
\end{proof}


\section{Global Considerations}

\subsection{Preliminaries}  Let $K$ be a number field, and let $M_K$ denote the set of places of $K$.  For each $v\in M_K$ we set the following local notation:
\begin{itemize}
	\item $K_v$ -- the completion of $K$ with respect to $v$.
	\item $\CC_v$ -- the completion of the algebraic closure of $K_v$; thus $\CC_v$ is both complete and algebraically closed.
	\item $|\cdot|_v$ -- the absolute value on $\CC_v$ whose restriction to $\QQ$ coincides with one of the usual real or $p$-adic absolute values.
	\item $r_v=[K_v:\QQ_v]/[K:\QQ]$ -- the ratio of local and global degrees.
\end{itemize}
More generally, all local notation introduced in $\S$~\ref{LocalSection} will now carry a subscript $v$ when attached to the pair $(\KK,|\cdot|)=(\CC_v,|\cdot|_v)$.  Given a line bundle $\Lcal$ on the projective line $\PP^1$ over $K$, an adelic metric $\|\cdot\|$ on $\Lcal$ is a family $\|\cdot\|=(\|\cdot\|_v)$, indexed by the places $v\in M_K$, where each $\|\cdot\|_v$ is a metric on $\Lcal$ over $\CC_v$.

With the absolute values $|\cdot|_v$ normalized as above, the product formula holds in the form $\sum_{v\in M_K}r_v\log|a|_v=0$ for each $a\in K^\times$.  If $K'/K$ is a finite extension, we have the local-global degree formula $\sum_{v'\mid v}r_{v'}=r_v$ for each place $v\in M_K$, where the sum is over the set of places $v'$ of $K'$ lying over $v$.  Fix once and for all homogeneous coordinates $(x_0:x_1)$ on the projective line $\PP^1$ over $K$.  Given a point $x=(x_0:x_1)\in\PP^1(K)$, define the standard height of $x$ by
\begin{equation}\label{StandardHeight}
h(x)=\sum_{v\in M_K}r_v\log\max\{|x_0|_v,|x_1|_v\}.
\end{equation}
By the product formula the value of $h(x)$ is invariant under replacing $(x_0:x_1)$ with $(ax_0:ax_1)$ for $a\in K^\times$, and the local-global degree formula shows that the value of $h(x)$ is unchanged upon replacing $K$ with a finite extension $K'/K$.  Thus $(\ref{StandardHeight})$ defines a $\Gal(\Kbar/K)$-invariant function $h:\PP^1(\Kbar)\to\RR$.

Alternatively, the standard height can be described in terms of the standard adelic metric $\|\cdot\|_\st=(\|\cdot\|_{\st,v})$ on the line bundle $\Ocal(1)$ over $K$.  Given a section $s\in\Gamma(\PP^1,\Ocal(1))$ defined over $K$ and a point $x\in\PP^1(K)\setminus\{\div(s)\}$, we have
\begin{equation}\label{StandardHeightAlt}
h(x)=\sum_{v\in M_K}r_v\log\|s(x)\|^{-1}_{\st,v}
\end{equation}
by $(\ref{StandardMetric})$, the product formula, and $(\ref{StandardHeight})$.

\subsection{Canonical heights and canonical adelic metrics}  Let $\varphi:\PP^1\to\PP^1$ be a rational map of degree $d\geq2$ defined over $K$.  The canonical height function $h_\varphi:\PP^1(\Kbar)\to\RR$ associated to $\varphi$ is the unique real-valued function on $\PP^1(\Kbar)$ satisfying the two properties:
\begin{itemize}
	\item The function $x\mapsto h(x)-h_\varphi(x)$ is bounded on $\PP^1(\Kbar)$;
	\item $h_\varphi(\varphi(x))=dh_\varphi(x)$ for all $x\in \PP^1(\Kbar)$.
\end{itemize}
Moreover, $h_\varphi$ is $\Gal(\Kbar/K)$-invariant and nonnegative, and $h_\varphi(x)=0$ if and only if $x$ is $\varphi$-preperiodic.  The existence, uniqueness, and basic properties of the canonical height were established by Call-Silverman \cite{CallSilverman} (see also Zhang \cite{Zhang}).

Given a $K$-polarization $\epsilon:\Ocal(d)\stackrel{\sim}{\to}\varphi^*\Ocal(1)$ of $\varphi$, we define the canonical adelic metric on $\Ocal(1)$ associated to the pair $(\varphi,\epsilon)$ to be the family $\|\cdot\|_{\varphi,\epsilon}=(\|\cdot\|_{\varphi,\epsilon,v})$ of canonical metrics over $\CC_v$ each place $v\in M_K$.  The following proposition is analogous the the decomposition $(\ref{StandardHeightAlt})$, giving a formula for the canonical height in terms of the canonical adelic metric.

\begin{prop}
Let $\varphi:\PP^1\to\PP^1$ be a rational map of degree $d\geq2$ defined over $K$, let $\epsilon:\Ocal(d)\stackrel{\sim}{\to}\varphi^*\Ocal(1)$ be a $K$-polarization, and let $s\in\Gamma(\PP^1,\Ocal(1))$ be a section defined over $K$.  Then
\begin{equation}\label{CanonicalLocalGlobal}
h_\varphi(x)=\sum_{v\in M_K}r_v\log\|s(x)\|^{-1}_{\varphi,\epsilon,v}
\end{equation}
for all $x\in\PP^1(K)\setminus\{\div(s)\}$.
\end{prop}
\begin{proof}
Denote by $h^*_\varphi(x)$ the right-hand-side of $(\ref{CanonicalLocalGlobal})$. The product formula shows that $h^*_\varphi(x)$ is independent of the section $s\in\Gamma(\PP^1,\Ocal(1))$, and the local-global degree formula shows that $h^*_\varphi(x)$ is invariant under replacing $K$ with a finite extension $K'/K$; therefore $(\ref{CanonicalLocalGlobal})$ defines a function $h^*_\varphi:\PP^1(\Kbar)\to\RR$.  In order to show that $h^*_\varphi=h_\varphi$, by the uniqueness of the canonical height it suffices to show that $h-h^*_\varphi$ is bounded on $\PP^1(\Kbar)$, and that $h_\varphi(\varphi(x))=dh_\varphi(x)$ for all $x\in \PP^1(\Kbar)$.

Let $S$ be a finite set of places of $K$ such that the pair $(\varphi,\epsilon)$ has good reduction at all places $v\not\in S$.  It  follows from Proposition~\ref{GoodReductionProp} that $\|\cdot\|_{\varphi,\epsilon,v}$ is equal to the standard metric $\|\cdot\|_{\st,v}$ for all $v\not\in S$.  We have
\begin{equation*}
|h(x)-h^*_\varphi(x)| = \bigg|\sum_{v\in S}r_v\log\frac{\|s(x)\|_{\varphi,\epsilon,v}}{\|s(x)\|_{\st,v}}\bigg| \leq \sum_{v\in S}r_vC_v,
\end{equation*}
where for each place $v$, $C_v=\sup|\log(\|s(x)\|_{\varphi,\epsilon,v}/\|s(x)\|_{\st,v})|$.  Note that if $K$ is replaced by a finite extension $K'/K$, and if we denote by $S'$ the set of places of $K'$ lying over $S$, then the local-global degree formula implies that $\sum_{v'\in S'}r_{v'}C_{v'}=\sum_{v\in S}r_vC_v$.  This shows that $h-h^*_\varphi$ is uniformly bounded on $\PP^1(\Kbar)$.

Now define $u=\epsilon^*\varphi^*s\in\Gamma(\PP^1,\Ocal(d))$.  Extending $K$ if necessary we can assume $u$ factors as $u=\otimes_{j=1}^{d}s_j$ for sections $s_j\in\Gamma(\PP^1,\Ocal(1))$ defined over $K$.  Using the identity $(\ref{CanMetricIdentityExp})$ at each place $v\in M_K$, we have
\begin{equation*}
h^*_\varphi(\varphi(x)) = \sum_{v\in M_K}r_v\log\|s(\varphi(x))\|^{-1}_{\varphi,\epsilon,v} = \sum_{j=1}^{d}\sum_{v\in M_K}r_v\log\|s_j(x)\|^{-1}_{\varphi,\epsilon,v} = dh^*_\varphi(x),
\end{equation*}
for each $x\in \PP^1(K)\setminus\{\div(s_1),\dots,\div(s_d)\}$.  Since the section $s$ is arbitrary, we conclude that $h_\varphi(\varphi(x))=dh_\varphi(x)$ for all $x\in \PP^1(\Kbar)$ as desired, completing the proof.
\end{proof}

\subsection{The global Arakelov-Zhang pairing}\label{GlobalAZPairingSect}

Let $\varphi:\PP^1\to\PP^1$ and $\psi:\PP^1\to\PP^1$ be two morpshisms of degree at least two defined over $K$.  The global Arakelov-Zhang pairing $\langle\varphi,\psi\rangle$ is defined as follows: select two sections $s,t\in\Gamma(\PP^1,\Ocal(1))$ defined over $K$, with $\div(s)\neq\div(t)$, and let
\begin{equation}\label{GlobalAZPairing}
\langle\varphi,\psi\rangle=\sum_{v\in M_K}r_v\langle\varphi,\psi\rangle_{s,t,v}+h_\varphi(\div(t))+h_\psi(\div(s)).
\end{equation}
It follows from $(\ref{SectionChange1})$, $(\ref{SectionChange2})$, and the product formula that the value of $(\ref{GlobalAZPairing})$ does not depend on the sections $s$ and $t$, and it follows from the local-global degree formula that the value of $(\ref{GlobalAZPairing})$ is invariant under replacing $K$ with a finite extension $K'/K$.  If $\epsilon_\varphi$ and $\epsilon_\psi$ are $K$-polarizations of $\varphi$ and $\psi$, respectively, then we have the alternate formulas
\begin{equation}\label{GlobalAZPairingAlt}
\begin{split}
\langle\varphi,\psi\rangle & = \sum_{v\in M_K}r_v\Big\{\langle\varphi,\psi\rangle_{s,t,v}-\log\|s(\div(t))\|_{\varphi,\epsilon_\varphi,v}-\log\|t(\div(s))\|_{\psi,\epsilon_\psi,v}\Big\} \\
	& = \sum_{v\in M_K}r_v\Big\{- \int\log\|s(x)\|_{\varphi,\epsilon_\varphi,v} d\mu_{\psi,v}(x)-\log\|t(\div(s))\|_{\psi,\epsilon_\psi,v}\Big\}.
\end{split}
\end{equation}
for $\langle\varphi,\psi\rangle$; the first of these follows by combining $(\ref{CanonicalLocalGlobal})$ and $(\ref{GlobalAZPairing})$, and the second follows from $(\ref{LocalAZPairing})$.

The following two results show that the Arakelov-Zhang pairing is symmetric and vanishes on the diagonal.

\begin{prop}\label{SymProp}
$\langle\varphi,\psi\rangle=\langle\psi,\varphi\rangle$.
\end{prop}
\begin{proof}
This follows at once from the definition $(\ref{GlobalAZPairing})$ and the symmetry $(\ref{LocalSymmetry})$ of the local Arakelov-Zhang pairing.
\end{proof}

\begin{prop}\label{DiagonalGlobal}
$\langle\varphi,\varphi\rangle=0$.
\end{prop}
\begin{proof}
Let $d=\deg(\varphi)$, and let $\epsilon$ be a $K$-polarization of $\varphi$.  Let $s\in\Gamma(\PP^1,\Ocal(1))$ be a section defined over $K$, and let $u=\epsilon^*\varphi^*s\in\Gamma(\PP^1,\Ocal(d))$.  Extending $K$ if necessary we may assume that $u$ factors as $u=\otimes_{j=1}^{d}s_j$ for sections $s_j\in\Gamma(\PP^1,\Ocal(1))$ defined over $K$.  Let $t\in\Gamma(\PP^1,\Ocal(1))$ be another section defined over $K$ and assume that $\div(s)\neq\div(t)$ and $\varphi(\div(s))\neq\div(t)$.

For each place $v$ of $K$ we have
\begin{equation*}
\langle\varphi,\varphi\rangle_{s,t,v} - \log\|s(\div(t))\|_{\varphi,\epsilon,v} = \sum_{j=1}^{d}\{\langle\varphi,\varphi\rangle_{s_j,t,v} - \log\|s_j(\div(t))\|_{\varphi,\epsilon,v}\}.
\end{equation*}
This follows at once from the invariance property $(\ref{CanMeasureInvariant})$ applied to the function $f(x)=\log\|s(x)\|_{\varphi,\epsilon,v}$, along with the definition $(\ref{LocalAZPairing})$ of the local Arakelov-Zhang pairing.  Summing over all places we have
\begin{equation*}
\begin{split}
\langle\varphi,\varphi\rangle -h_\varphi(\div(s)) & = \sum_{v\in M_K}r_v\langle\varphi,\varphi\rangle_{s,t,v}+h_\varphi(\div(t)) \\
	& = \sum_{v\in M_K}r_v\Big\{\langle\varphi,\varphi\rangle_{s,t,v}- \log\|s(\div(t))\|_{\varphi,\epsilon,v}\Big\} \\
	& = \sum_{j=1}^{d}\sum_{v\in M_K}r_v\Big\{\langle\varphi,\varphi\rangle_{s_j,t,v}- \log\|s_j(\div(t))\|_{\varphi,\epsilon,v}\Big\} \\
	& = \sum_{j=1}^{d}\Big\{\sum_{v\in M_K}r_v\langle\varphi,\varphi\rangle_{s_j,t,v}+h_\varphi(\div(t))\Big\} \\
	& = d\langle\varphi,\varphi\rangle -\sum_{j=1}^{d}h_\varphi(\div(s_j)).
\end{split}
\end{equation*}
Since $h_\varphi(\div(s))=h_\varphi(\varphi(\div(s_j)))=dh_\varphi(\div(s_j))$ for each $j$, we deduce that $\langle\varphi,\varphi\rangle=d\langle\varphi,\varphi\rangle$, whereby $\langle\varphi,\varphi\rangle=0$.
\end{proof}

\subsection{Equidistribution}  Let $\varphi:\PP^1\to\PP^1$ and $\psi:\PP^1\to\PP^1$ be two rational maps, each of degree at least two, defined over $K$.  In this section we will use dynamical equidistribution results to give a useful alternative expression for the global Arakelov-Zhang pairing $\langle\varphi,\psi\rangle$.

Let $\{Z_n\}$ be a sequence of finite $\Gal(\Kbar/K)$-invariant multisets in $\PP^1(\Kbar)$.  We say the sequence $\{Z_n\}$ is $h_\psi$-small if
\begin{equation*}
\frac{1}{|Z_n|}\sum_{x\in Z_n}h_\psi(x) \to 0.
\end{equation*}
Throughout this paper, sums over $x$ in a multiset $Z$, and its cardinality $|Z|$, are computed according to multiplicity.

Let $\mu_\psi=(\mu_{\psi,v})$ be the collection, indexed by the set of places $v$ of $K$, of canonical measures on $\Psf^1_v$ with respect to $\psi$.  We say a sequence $\{Z_n\}$ of $\Gal(\Kbar/K)$-invariant multisets in $\PP^1(\Kbar)$ is $\mu_\psi$-equidistributed if
\begin{equation}\label{EquidistDef}
\frac{1}{|Z_n|}\sum_{x\in Z_n}f(x) \to \int_{\Psf_v^1}f(x)d\mu_{\psi,v}(x)
\end{equation}
for each place $v\in M_K$ and each continuous function $f:\Psf^1_v\to\RR$.  In the left-hand-side of $(\ref{EquidistDef})$, the function $f:\Psf^1_v\to\RR$ is evaluated at each $x\in Z_n$ via some fixed embedding $\Kbar\hookrightarrow\CC_v$; since each multiset $Z_n$ is $\Gal(\Kbar/K)$-invariant, the value of the left-hand-side of $(\ref{EquidistDef})$ is independent of the choice of embedding.

\begin{ex}\label{GaloisOrbitsEx}
Let $\{x_n\}$ be a sequence of distinct points in $\PP^1(\Kbar)$ such that $h_\psi(x_n)\to0$, and for each $n$ let $Z_n$ be the set of $\Gal(\Kbar/K)$ conjugates of $x_n$ in $\PP^1(\Kbar)$.  Since $h_\psi$ is $\Gal(\Kbar/K)$-invariant, each point $x\in Z_n$ has $\psi$-canonical height $h_\psi(x)=h_\psi(x_n)$.  In particular, the sequence $\{Z_n\}$ is $h_\psi$-small.  The dynamical equidistribution theorem for Galois orbits, proved independently by Baker-Rumely \cite{BakerRumely}, Chambert-Loir \cite{ChambertLoir}, and Favre-Rivera-Letelier \cite{FavreRiveraLetelier}, states that $\{Z_n\}$ is $\mu_\psi$-equidistributed as well.
\end{ex}

\begin{ex}\label{PeriodicPointsEx}
Let $Z_n$ be the multiset of $\psi$-periodic points in $\PP^1(\Kbar)$ of period $n$; that is, the points $x$ satisfying $\psi^n(x)=x$.  Then $|Z_n|=\deg(\psi)^n+1$, and since all $\psi$-periodic points have $\psi$-canonical height zero, the sequence $\{Z_n\}$ is $h_\psi$-small.  Moreover, the sequence $\{Z_n\}$ is $\mu_\psi$-equidistributed, as shown by Ljubich \cite{Ljubich} at the archimedean places and by Favre-Rivera-Letelier \cite{FavreRiveraLetelier2} at the non-archimedean places.
\end{ex}

\begin{thm}\label{MainTheorem}
Let $\{Z_n\}$ be a sequence of finite $\Gal(\Kbar/K)$-invariant multisets in $\PP^1(\Kbar)$. If $\{Z_n\}$ is $h_\psi$-small and $\mu_\psi$-equidistributed, then
\begin{equation}\label{MainTheoremIdentity}
\frac{1}{|Z_n|}\sum_{x\in Z_n}h_\varphi(x) \to \langle\varphi,\psi\rangle.
\end{equation}
\end{thm}
\begin{proof}
Let $\epsilon_\varphi$ and $\epsilon_\psi$ be $K$-polarizations of $\varphi$ and $\psi$, respectively, and let $s,t\in\Gamma(\PP^1,\Ocal(1))$ be sections defined over $K$ with $\div(s)\neq\div(t)$.  By Propositions~\ref{GoodReductionProp} and \ref{GoodReductionAZProp} there exists a finite set $S$ of places, including all of the archimedean places, such that the following conditions hold for all places $v\not\in S$:
\begin{itemize}
	\item $\|\cdot\|_{\varphi,\epsilon_\varphi,v}=\|\cdot\|_{\psi,\epsilon_\psi,v}=\|\cdot\|_{\st,v}$;
	\item $\langle\varphi,\psi\rangle_{s,t,v} = \langle\psi,\psi\rangle_{s,t,v}=0$.
\end{itemize}
For each place $v$ of $K$ and each $x\in \PP^1(\KK)\setminus\{\div(s)\}$, define
\begin{equation*}
f_v(x)=\log\frac{\|s(x)\|_{\psi,\epsilon_\psi,v}}{\|s(x)\|_{\varphi,\epsilon_\varphi,v}}.
\end{equation*}
Note that $f_v$ extends to a continuous function $f_v:\Psf^1_v\to\RR$, and that if $v\not\in S$ then $f_v$ is identically zero.

We have
\begin{equation}\label{MainTheoremIdentity2}
\begin{split}
\lim_{n\to+\infty}\frac{1}{|Z_n|}\sum_{x\in Z_n}(h_\varphi(x)-h_\psi(x)) & = \lim_{n\to+\infty}\frac{1}{|Z_n|}\sum_{x\in Z_n}\sum_{v\in S}r_vf_v(x) \\
	& = \sum_{v\in S}r_v\lim_{n\to+\infty}\frac{1}{|Z_n|}\sum_{x\in Z_n}f_v(x) \\
	& = \sum_{v\in S}r_v\int_{\Psf_v^1}f_v(x)d\mu_{\psi,v}(x) \\
	& = \sum_{v\in S}r_v\Big\{ \langle\varphi,\psi\rangle_{s,t,v} - \langle\psi,\psi\rangle_{s,t,v}+f_v(\div(t)) \Big\} \\
	& = \langle\varphi,\psi\rangle-\langle\psi,\psi\rangle.
\end{split}
\end{equation}
In the preceding calculation, the first equality uses the formula $(\ref{CanonicalLocalGlobal})$ and the fact that $f_v$ vanishes identically when $v\not\in S$.  The second equality interchanges the limit with the sum over places $v\in S$; this is justified because the finite set $S$ is independent of the parameter $n$.  The third equality uses the fact that $\{Z_n\}$ is $\mu_\psi$-equidistributed.  The fourth equality uses the definition $(\ref{LocalAZPairing})$ of the local Arakelov-Zhang pairing.  The final equality follows from $(\ref{GlobalAZPairingAlt})$ and the fact that the the summand vanishes at all places $v\not\in S$.

Since the sequence $\{Z_n\}$ is $h_\psi$-small and $\langle\psi,\psi\rangle=0$, the calculation $(\ref{MainTheoremIdentity2})$ implies $(\ref{MainTheoremIdentity})$.
\end{proof}

\begin{cor}\label{EquidistCor1}
Let $\{x_n\}$ be a sequence of distinct points in $\PP^1(\Kbar)$ such that $h_\psi(x_n)\to0$.  Then $h_\varphi(x_n)\to\langle\varphi,\psi\rangle$.
\end{cor}
\begin{proof}
Letting $Z_n$ denote the set of $\Gal(\Kbar/K)$-conjugates of $x_n$, as explained in Example~\ref{GaloisOrbitsEx} the sequence $\{Z_n\}$ is $h_\psi$-small and $\mu_\psi$-equidistributed.  Since $h_\varphi$ is $\Gal(\Kbar/K)$-invariant, by Theorem~\ref{MainTheorem} we have $h_\varphi(x_n)=\frac{1}{|Z_n|}\sum_{x\in Z_n}h_\varphi(x)\to\langle\varphi,\psi\rangle$.
\end{proof}

\begin{cor}\label{EquidistCor2}
The limit in $(\ref{SzpiroTuckerPairing})$ exists and $\Sigma(\varphi,\psi)=\langle\varphi,\psi\rangle$.
\end{cor}
\begin{proof}
Letting $Z_n$ be the multiset of $\psi$-periodic points in $\PP^1(\Kbar)$ of period $n$, as explained in Example~\ref{PeriodicPointsEx} the sequence $\{Z_n\}$ is $h_\psi$-small and $\mu_\psi$-equidistributed.  It follows from Theorem~\ref{MainTheorem} that $\Sigma(\varphi,\psi)=\frac{1}{|Z_n|}\sum_{x\in Z_n}h_\varphi(x)\to\langle\varphi,\psi\rangle$.
\end{proof}

\begin{cor}\label{EquidistCor3}
The Arakelov-Zhang pairing $\langle\varphi,\psi\rangle$ is nonnegative.  Moreover, the following five conditions are equivalent:
\begin{quote}
{\bf (a)}  $\langle\varphi,\psi\rangle=0$; \\
{\bf (b)}  $h_\varphi=h_\psi$; \\
{\bf (c)}  $\PrePer(\varphi)=\PrePer(\psi)$; \\
{\bf (d)}  $\PrePer(\varphi)\cap\PrePer(\psi)$ is infinite; \\
{\bf (e)}  $\liminf_{x\in\PP^1(\Kbar)}(h_\varphi(x)+h_\psi(x))=0$.
\end{quote}
\end{cor}
\begin{proof}
The nonnegativity of $\langle\varphi,\psi\rangle$ follows from Corollary~\ref{EquidistCor1} and the nonnegativity of the canonical height $h_\varphi$.  It also follows immediately from Corollary~\ref{EquidistCor1} that {\bf (a)} and {\bf (e)} are equivalent.  To complete the proof we will show that {\bf (b)} $\Rightarrow$ {\bf (c)} $\Rightarrow$ {\bf (d)} $\Rightarrow$ {\bf (e)} $\Rightarrow$ {\bf (b)}.  Since $h_\varphi$ vanishes precisely on $\PrePer(\varphi)$, and similarly for $\psi$, it follows that {\bf (b)} $\Rightarrow$ {\bf (c)}.  Since $\PrePer(\varphi)$ is infinite, it follows that {\bf (c)} $\Rightarrow$ {\bf (d)}.  Since preperiodic points have height zero, it follows that {\bf (d)} $\Rightarrow$ {\bf (e)}.

It remains only to show that {\bf (e)} $\Rightarrow$ {\bf (b)}.  The condition {\bf (e)} means that there exists a sequence $\{x_n\}$ of distinct points in $\PP^1(\Kbar)$ such that $h_\varphi(x_n)\to0$ and $h_\psi(x_n)\to0$.  Letting $Z_n$ denote the set of $\Gal(\Kbar/K)$-conjugates of $x_n$, since $h_\varphi(x_n)\to0$ and $h_\psi(x_n)\to0$ it follows from Example~\ref{GaloisOrbitsEx} that $\{Z_n\}$ is both $\mu_\varphi$-equidistributed and $\mu_\psi$-equidistributed.  From the definition $(\ref{EquidistDef})$ it follows that $\int_{\Psf_v^1}f(x)d\mu_{\varphi,v}(x)=\int_{\Psf_v^1}f(x)d\mu_{\psi,v}(x)$ for each place $v$ of $K$ and each continuous function $f:\Psf^1_v\to\RR$.  This means that $\mu_{\varphi,v}=\mu_{\psi,v}$ as measures on $\Psf^1_v$.

Now let $\epsilon_\varphi$ and $\epsilon_\psi$ be $K$-polarizations of $\varphi$ and $\psi$, respectively, and let $s\in\Gamma(\PP^1,\Ocal(1))$ be a section defined over $K$.  By $(\ref{CanMeasureDef})$ and the equality of measures $\mu_{\varphi,v}=\mu_{\psi,v}$, along with the fact that $\Delta$ vanishes precisely on the constant functions, we deduce that
\begin{equation*}
-\log\|s(x)\|_{\varphi,\epsilon_\varphi,v}=-\log\|s(x)\|_{\psi,\epsilon_\psi,v} + c_v
\end{equation*}
for each place $v$ of $K$, where $c_v\in\RR$ is a constant which does not depend on the section $s$.  Summing over all places and using $(\ref{CanonicalLocalGlobal})$ we deduce that
\begin{equation}\label{EqualUpToConstant}
h_\varphi(x)=h_\psi(x) + c,
\end{equation}
for all $x\in \PP^1(K)\setminus\div(s)$, where $c=\sum_{v\in M_K}r_vc_v$.  Since all of the terms in $(\ref{EqualUpToConstant})$ are independent of the section $s$ and invariant under replacing $K$ with a finite extension $K'/K$, it follows that $(\ref{EqualUpToConstant})$ holds for all $x\in \PP^1(\Kbar)$.  Taking $x$ to be $\varphi$-preperiodic, we deduce that $c=-h_\psi(x)\leq0$.  On the other hand, if $x$ is $\psi$-preperiodic then $c=h_\varphi(x)\geq0$.  We conclude that $c=0$, and therefore $h_\varphi=h_\psi$ as desired.
\end{proof}


\section{A Height Difference Bound}\label{HeightDiffSection}

Let $\sigma:\PP^1\to\PP^1$ be the map defined by $\sigma(x)=x^2$, or in homogeneous coordinates $\sigma(x_0:x_1)=(x_0^2:x_1^2)$.  Let $\epsilon_\sigma:\Ocal(2)\stackrel{\sim}{\to}\sigma^*\Ocal(1)$ be the polarization which is normalized so that $\epsilon_\sigma^*\sigma^*x_j=x_j^2$.  In this case the canonical adelic metric $\|\cdot\|_{\sigma,\epsilon_\sigma}$ on $\Ocal(1)$ is the same as the standard adelic metric $\|\cdot\|_{\st}$.  It then follows from $(\ref{StandardHeightAlt})$ and $(\ref{CanonicalLocalGlobal})$ that the canonical height $h_{\sigma}$ is the same as the standard height $h_\st$.

Now consider an arbitrary rational map $\psi:\PP^1\to\PP^1$ of degree at least two.  Using Theorem~\ref{IntroThm3} and a result of Szpiro-Tucker \cite{SzpiroTucker}, we can give an upper bound on the difference between the canonical height $h_\psi$ and the standard height $h_\st$ in terms of the pairing $\langle\sigma,\psi\rangle$.

\begin{thm}\label{HeightDiffProp}
Let $\sigma:\PP^1\to\PP^1$ be the map defined by $\sigma(x)=x^2$, and let $\psi:\PP^1\to\PP^1$ be an arbitrary map of degree $d\geq2$ defined over a number field $K$.  Then
\begin{equation}\label{HeightDiffIneq}
h_\psi(x)-h_\st(x)\leq \langle\sigma,\psi\rangle + h_\psi(\infty) +\log 2
\end{equation}
for all $x\in\PP^1(\Kbar)$.
\end{thm}

\begin{proof}
Extending $K$ if necessary we may assume without loss of generality that $x\in\PP^1(K)$, and since the inequality $(\ref{HeightDiffIneq})$ obviously holds when $x=\infty$, we may assume that $x\neq\infty$.  For each place $v$ of $K$ and each integer $n\geq1$ define
\begin{equation*}
S_{v,n} = \frac{1}{d^n+1}\sum_{\stackrel{\psi^n(\alpha)=\alpha}{\alpha\neq\infty}}\log|x-\alpha|_v
\end{equation*}
It follows from \cite{SzpiroTucker} Thm. 4.10 that the limit $\lim_{n\to+\infty}S_{v,n}$ exists for all places $v$ of $K$, that this limit vanishes for all but finitely many places $v$ of $K$, and that
\begin{equation*}
h_\psi(x)-h_\psi(\infty)=\sum_{v\in M_K}r_v\lim_{n\to+\infty}S_{v,n}.
\end{equation*}
Define constants $\theta_v=\log 2$ if $v$ is archimedean, and $\theta_v=0$ if $v$ is non-archimedean.  Then we have the elementary inequality
\begin{equation*}
\log|x-\alpha|_v\leq\log^+|x|_v+\log^+|\alpha|_v+\theta_v,
\end{equation*}
and therefore using Theorem~\ref{IntroThm3} we have
\begin{equation*}
\begin{split}
h_\psi(x)-h_\psi(\infty) & \leq \sum_{v\in M_K}r_v\limsup_{n\to+\infty}\bigg\{\log^+|x|_v+\frac{1}{d^n+1}\sum_{\stackrel{\psi^n(\alpha)=\alpha}{\alpha\neq\infty}}\log^+|\alpha|_v+\theta_v\bigg\} \\
	& \leq h_\st(x)+ \frac{1}{d^n+1}\sum_{\psi^n(\alpha)=\alpha}h_\st(\alpha)+\log 2 \\
	& = h_\st(x)+ \langle\sigma,\psi\rangle +\log 2,
\end{split}
\end{equation*}
which implies $(\ref{HeightDiffIneq})$.
\end{proof}

\noindent
{\em Remark.}  We do not know whether the constant $\log 2$ on the right-hand-side of $(\ref{HeightDiffIneq})$ is best possible.  However, in $\S$~\ref{COCSect} we will give an example which shows that Theorem~\ref{HeightDiffProp} is false if $\log 2$ is replaced by a sufficiently small positive constant.


\section{Examples}\label{ExampleSection}

\subsection{The squaring map and an arbitrary map}\label{SquaringSect}  Let $\sigma:\PP^1\to\PP^1$ be the map defined by $\sigma(x)=x^2$, and let $\psi:\PP^1\to\PP^1$ be an arbitrary rational map of degree at least two.  The following proposition will be useful in calculating (or at least estimating) the value of $\langle\sigma,\psi\rangle$ in several specific cases.  As usual we denote by $x=(x:1)$ the affine coordinate on $\PP^1$, where $\infty=(1:0)$.  For $r\geq0$, define $\log^+r=\log\max\{1,r\}$.

\begin{prop}\label{ExampleProp}
Let $\sigma:\PP^1\to\PP^1$ be the map defined by $\sigma(x)=x^2$, and let $\psi:\PP^1\to\PP^1$ be an arbitrary map of degree at least two defined over a number field $K$.  Then
\begin{equation}\label{ExampleIdentity}
\langle\sigma,\psi\rangle=h_\psi(\infty)+\sum_{v\in M_K}r_v\int\log^+|x|_v d\mu_{\psi,v}(x).
\end{equation}
\end{prop}
\begin{proof}
Let $\epsilon_\sigma:\Ocal(2)\stackrel{\sim}{\to}\sigma^*\Ocal(1)$ be the polarization which is normalized so that $\epsilon_\sigma^*\sigma^*x_j=x_j^2$, and let $s\in\Gamma(\PP^1,\Ocal(1))$ be the section defined by $s(x_0,x_1)=x_1$; thus $\div(s)=\infty$.  Let $t\in\Gamma(\PP^1,\Ocal(1))$ be any section with $\div(t)\neq\infty$.  At each place $v\in M_K$ we have
\begin{equation*}
-\log\|s(x)\|_{\sigma,\epsilon_\sigma,v}=-\log\|s(x)\|_{\st,v}=-\log(1/\max\{1,|x|_v\})=\log^+|x|_v.
\end{equation*}
The identity $(\ref{ExampleIdentity})$ follows from this along with $(\ref{CanonicalLocalGlobal})$ and $(\ref{GlobalAZPairingAlt})$.
\end{proof}

\noindent
{\em Remark.}
In $(\ref{ExampleIdentity})$ at the nonarchimedean places, the integrand $x\mapsto\log^+|x|_v$ should be interpreted as a function on the whole Berkovich projective line $\Psf_v^1$.  The extension of the function $x\mapsto|x|_v$ to all of $\Psf^1_v$ is described in the last paragraph of $\S$~\ref{BerkovichSect}.

\subsection{The squaring map after a certain affine change of coordinates}\label{COCSect}  Let $\sigma:\PP^1\to\PP^1$ be the squaring map $\sigma(x)=x^2$, let $\alpha\in K$, and let $\sigma_\alpha:\PP^1\to\PP^1$ be the map defined by $\sigma_\alpha(x)=\alpha-(\alpha-x)^2$.  In other words, $\sigma_\alpha=\gamma_\alpha^{-1}\circ\sigma\circ\gamma_\alpha$, where $\gamma_\alpha$ is the automorphism of $\PP^1$ given by $\gamma_\alpha(x)=\alpha-x$.

In order to calculate the value of $\langle\sigma,\sigma_\alpha\rangle$ we must introduce the real-valued function defined for $t\geq0$ by $I(t)=\int_0^1\log^+|t+e^{2\pi i\theta}| d\theta-\log^+t$.

\begin{lem}\label{ILemma}
The function $I(t)$ is continuous, nonnegative, monotone increasing on the interval $0<t<1$, and monotone decreasing on the interval $1<t<2$.  Moreover $I(t)=0$ when $t=0$ and when $t\geq2$.  In particular, $\sup_{t\geq0}I(t)=I(1)$.
\end{lem}
\begin{proof}
Clearly $I(0)=0$.  Using Jensen's formula $\int_0^1\log|t+e^{2\pi i\theta}| d\theta=\log^+t$ and the identity $\log r=\log^+r-\log^+(1/r)$, we have $I(t)=\int_0^1\log^+(1/|t+e^{2\pi i\theta}|) d\theta$, from which it follows that $I(t)$ is nonnegative and vanishes for $t\geq2$.  When $0<t<1$ the integrand $\log^+|t+e^{2\pi i\theta}|$ is monotone increasing as a function of $t$, and when $1<t<2$ the integrand $\log^+(1/|t+e^{2\pi i\theta}|)$ is monotone decreasing as a function of $t$; thus $I(t)$ has these same properties.
\end{proof}

We will now evaluate the pairing $\langle\sigma,\sigma_\alpha\rangle$ in terms of the height $h_\st(\alpha)$ and the function $I(t)$, and we will give bounds for the pairing $\langle\sigma,\sigma_\alpha\rangle$ in terms of $h_\st(\alpha)$ and the constant $I(1)$.  These bounds are sharp in the sense that there exist cases of equality, and thus it may be desirable to calculate $I(1)$ explicitly.  It turns out that
\begin{equation}\label{SmythCalc}
I(1) = \int_{0}^{1}\log^+|1+e^{2\pi i\theta}| d\theta=\int_{-1/3}^{1/3}\log|1+e^{2\pi i\theta}| d\theta = \frac{3\sqrt{3}}{4\pi}L(2,\chi)\approx 0.323067...,
\end{equation}
where $\chi$ is the nontrivial quadratic character modulo $3$ (that is $\chi(n)=0,1,-1$ according to whether $n\equiv0,1,2\pmod{3}$, repectively), and $L(2,\chi)=\sum_{n\geq1}\chi(n)n^{-2}$ is the value at $s=2$ of the associated Dirichlet $L$-function.  The calculation $(\ref{SmythCalc})$ follows from expanding $\log|1+e^{2\pi i\theta}|$ into its Fourier series and integrating term-by-term; it is equivalent to Smyth's evaluation of the Mahler measure of the two-variable polynomial $1+x+y$; see \cite{Boyd}.

\begin{prop}\label{ExamplePropCOC}
Let $\alpha\in K$, let $\sigma(x)=x^2$, and let $\sigma_\alpha(x)=\alpha-(\alpha-x)^2$.  Then:
\begin{quote}
{\bf (a)} $\langle\sigma,\sigma_\alpha\rangle= h_\st(\alpha)+\sum_{v\mid\infty}r_vI(|\alpha|_v)$. \\
{\bf (b)} $h_\st(\alpha)\leq\langle\sigma,\sigma_\alpha\rangle\leq h_\st(\alpha)+\frac{3\sqrt{3}}{4\pi}L(2,\chi)$. \\
{\bf (c)} If $|\alpha|_v\geq2$ for all archimedean $v\in M_K$, then $\langle\sigma,\sigma_\alpha\rangle= h_\st(\alpha)$. \\
{\bf (d)}  $\langle\sigma,\sigma_1\rangle=\frac{3\sqrt{3}}{4\pi}L(2,\chi)$.
\end{quote}
\end{prop}
\begin{proof}
Let $v$ be a place of $K$.  The automorphism $\gamma_\alpha:\Psf_v^1\to\Psf_v^1$ has order $2$, and it interchanges the unit disc $B_v(0,1)$ of $\CC_v$ with the disc $B_v(\alpha,1)$ in $\CC_v$ centered at $\alpha$ with radius $1$.  When $v$ is non-archimedean, this implies that $\gamma_\alpha$ interchanges the two points $\zeta_{0,1}$ and $\zeta_{\alpha,1}$ of $\Psf^1$ (note that $\zeta_{0,1}=\zeta_{\alpha,1}$ when $|\alpha|_v\leq1$).  The canonical measures associated to $\sigma$ and $\sigma_\alpha$ are related by the identity $\mu_{\sigma_\alpha,v}(x)=\mu_{\sigma,v}(\gamma_\alpha(x))$.  Since $\mu_{\sigma,v}$ is the standard measure $\mu_{\st,v}$ described in $\S$~\ref{StandardMeasureSect}, this means that when $v$ is archimedean, $\mu_{\sigma_\alpha,v}$ is the uniform unit measure supported on the circle $|x-\alpha|_v=1$ in $\CC_v$, and when $v$ is non-archimedean, $\mu_{\sigma_\alpha,v}$ is the Dirac measure supported on the point $\zeta_{\alpha,1}$ of $\Psf^1$.  For non-archimedean $v$ we deduce that
\begin{equation*}
\int\log^+|x|_v d\mu_{\sigma_\alpha,v}(x)=\log^+|\zeta_{\alpha,1}|_v=\log^+|\alpha|_v,
\end{equation*}
and for archimedean $v$, we have
\begin{equation*}
\int\log^+|x|_v d\mu_{\sigma_\alpha,v}(x)=\int_0^1\log^+||\alpha|_v+e^{2\pi it}| dt=\log^+|\alpha|_v+I(|\alpha|_v).
\end{equation*}

The proof of {\bf (a)} is completed by assembling these calculations together into a global identity using $(\ref{ExampleIdentity})$, and noting that $h_{\sigma_\alpha}(\infty)=0$ since $\sigma_\alpha$ fixes $\infty$.  {\bf (b)} follows from {\bf (a)} along with the fact that $0\leq I(t)\leq I(1)=\frac{3\sqrt{3}}{4\pi}L(2,\chi)$ for all $t\geq0$, and that $\sum_{v\mid\infty}r_v=1$.  {\bf (c)} follows from {\bf (a)} and the fact that $I(t)=0$ for all $t\geq2$.  {\bf (d)} follows from {\bf (a)} with $K=\QQ$ and $\alpha=1$, and the identity $(\ref{SmythCalc})$.
\end{proof}

\noindent
{\em Remark.}  Since $\sigma_\alpha=\gamma_\alpha^{-1}\circ\sigma\circ\gamma_\alpha$, it follows from basic properties of canonical height functions that $h_{\sigma_\alpha}(x)=h_\st(\gamma_\alpha(x))=h_\st(\alpha-x)$.  Taking $\psi=\sigma_1$ in Theorem~\ref{HeightDiffProp}, and applying Proposition~\ref{ExamplePropCOC}~{\bf (d)}, the inequality $(\ref{HeightDiffIneq})$ becomes
\begin{equation*}
h_\st(1-x)-h_\st(x) \leq \frac{3\sqrt{3}}{4\pi}L(2,\chi)+\log2.
\end{equation*}
Taking $x=-1$ shows that in $(\ref{HeightDiffIneq})$, the constant $\log 2$ cannot be replaced with a constant which is less than $\log 2-\frac{3\sqrt{3}}{4\pi}L(2,\chi)\approx 0.37008...$

\subsection{Quadratic polynomials}\label{QuadPolySect}  Let $\sigma:\PP^1\to\PP^1$ be the squaring map $\sigma(x)=x^2$, let $c\in K$, and let $\psi_c:\PP^1\to\PP^1$ be the map defined by $\psi_c(x)=x^2+c$.  We will prove upper and lower bounds on the pairing $\langle\sigma,\psi_c\rangle$ which show that $\langle\sigma,\psi_c\rangle$ is about the size of $(1/2)h_\st(c)$ when $h_\st(c)$ is large.

The proof is based on the fact that the support of the canonical measure $\mu_{\psi_c,v}$ is fairly well understood.  If $v$ is non-archimedean and $|c|_v$ is small, then (at least some conjugate of) $\psi_c$ has good reduction and so $\mu_{\psi_c,v}$ is supported on a certain point of $\Psf^1_v$.  For all other $c$ in the non-archimedean case, and for any $c$ in the archimedean case, $\supp(\mu_{\psi_c,v})$ is contained in the filled Julia set $J_v(\psi_c)=\{x\in\CC_v\mid|\psi_c^k(x)|_v\not\to+\infty\}$; in the non-archimedean case these facts are proved in \cite{BenedettoBriendPerdry}; in the archimedean case see \cite{Milnor}.  Finally,  elementary arguments show that $J_v(\psi_c)$ must itself be contained in a disc or an annulus with radii comparable to $\max\{1,|c|_v\}^{1/2}$.

\begin{prop}\label{ExamplePropQuad}
Let $c\in K$, let $\sigma(x)=x^2$, and let $\psi_c(x)=x^2+c$.  Then
\begin{equation}\label{ExampleQuadIneq}
(1/2)h_\st(c)-\log3 \leq\langle\sigma,\psi_c\rangle\leq (1/2)h_\st(c)+\log2.
\end{equation}
\end{prop}
\begin{proof}
We will first show that if $v$ is non-archimedean then
\begin{equation}\label{QuadNonArch}
\int\log^+|x|_v d\mu_{\psi_c,v}(x)=(1/2)\log^+|c|_v.
\end{equation}
First suppose that $v\nmid2$.  If $|c|_v\leq1$, then $\psi_c$ has good reduction, so $\mu_{\psi_c,v}$ is the standard measure $\mu_{\st,v}$, that is, the Dirac measure supported at the point $\zeta_{0,1}$ of $\Psf^1_v$.  In this case both sides of $(\ref{QuadNonArch})$ vanish.  If $|c|_v>1$, then
\begin{equation}\label{QuadNonArchSupp}
\supp(\mu_{\psi_c,v}) \subseteq J_v(\psi_c)\subseteq\{x\in\CC_v\mid |x|_v=|c|_v^{1/2}\}.
\end{equation}
The second inclusion follows from the fact that when $|x|_v\neq|c|_v^{1/2}$, the ultrametric inequality implies that $|\psi_c^k(x)|_v\to+\infty$.  It follows from $(\ref{QuadNonArchSupp})$ that the integrand on the left-hand-side of $(\ref{QuadNonArch})$ is the constant function $(1/2)\log^+|c|_v$, completing the proof of $(\ref{QuadNonArch})$ when $v$ is non-archimedean and $v\nmid2$.

Suppose now that $v\mid2$.  If $|c|_v\leq1$ or if $|c|_v>4$, then the proof of $(\ref{QuadNonArch})$ is the same as the case $v\nmid2$.  However, if $1<|c|_v\leq4$, then consider the conjugate $\tilde{\psi}_c=\gamma^{-1}\circ\psi_c\circ\gamma$ of $\psi_c$, where $\gamma(x)=x+b$ and $b\in\CC_v$ is a fixed point of $\psi_c$.  The ultrametric inequality implies that $|b|_v=|c|_v^{1/2}\leq2$, and we calculate $\tilde{\psi}_c(x)=x^2+2bx$.  Since $|b|_v\leq2$, we must have $|2b|_v\leq1$, showing that $\tilde{\psi}_c$ has good reduction.  Therefore $\mu_{\tilde{\psi}_c,v}$ is supported on the point $\zeta_{0,1}$ of $\Psf^1_v$.  As an automorphism of $\Psf_v^1$, $\gamma$ takes $\zeta_{0,1}$ to $\zeta_{b,1}$, and it follows that $\mu_{\psi_c,v}$ is supported on $\zeta_{b,1}$ in $\Psf_v^1$.  The left-hand-side of $(\ref{QuadNonArch})$ is therefore equal to $\log|\zeta_{b,1}|_v=\log|b|_v=(1/2)\log|c|_v$, completing the proof of $(\ref{QuadNonArch})$.

We will now show that if $v$ is archimedean then
\begin{equation}\label{QuadArch}
(1/2)\log^+|c|_v-\log3\leq \int\log^+|x|_v d\mu_{\psi_c,v}(x)\leq (1/2)\log^+|c|_v+\log2.
\end{equation}
The upper bound in $(\ref{QuadArch})$ follows immediately from the fact that
\begin{equation}\label{QuadArchSupp1}
\supp(\mu_{\psi_c,v}) \subseteq J_v(\psi_c)\subseteq\{x\in\CC_v\mid |x|_v\leq B\}.
\end{equation}
where $B=2\max\{1,|c|_v\}^{1/2}$.  To prove the second inclusion in $(\ref{QuadArchSupp1})$, it suffices to show that if $|x|_v>B$, then $|\psi_c^k(x)|_v\to+\infty$.  Assuming that $|x|_v>B$, we have $|x^2|_v>4|c|_v$, so $|\psi_c(x)-x^2|_v=|c|_v<\frac{1}{4}|x^2|_v$.  This implies that $|\psi_c(x)|_v>\frac{3}{4}|x^2|_v>\frac{3}{2}|x|_v$, where the last inequality follows from the fact that $|x|_v>2$.  Iterating we deduce that $|\psi_c^k(x)|_v>(\frac{3}{2})^k|x|_v\to+\infty$, as desired.

We turn to the lower bound in $(\ref{QuadArch})$.  If $|c|_v\leq9$ then there is nothing to prove, since the left-hand-side of $(\ref{QuadArch})$ is nonpositive and the integral in $(\ref{QuadArch})$ is nonnegative.  So we may assume that $|c|_v>9$.  The desired inequality follows immediately from the fact that
\begin{equation}\label{QuadArchSupp2}
\supp(\mu_{\psi_c,v}) \subseteq J_v(\psi_c)\subseteq\{x\in\CC_v\mid |x|_v\geq A\}.
\end{equation}
where $A=\frac{1}{3}|c|_v^{1/2}$.  To prove the second inclusion in $(\ref{QuadArchSupp2})$, it suffices to show that if $|x|_v<A$, then $|\psi_c^k(x)|_v\to+\infty$.  Assuming that $|x|_v<A$, we have
\begin{equation}\label{QuadArchIneq}
|\psi_c(x)|_v=|x^2+c|_v\geq|c|_v-|x^2|_v>|c|_v-(1/9)|c|_v=(8/9)|c|_v>(8/3)|c|_v^{1/2},
\end{equation}
where in the last inequality we used that $|c|_v^{1/2}>3$.  In particular, $(\ref{QuadArchIneq})$ implies that $\psi_c(x)$ is not in the disc on the right-hand-side of $(\ref{QuadArchSupp1})$, which implies that $\psi_c(x)\notin J_v(\psi_c)$.  Therefore $|\psi_c^k(x)|_v\to+\infty$, as desired.

Finally, the proof of $(\ref{ExampleQuadIneq})$ is completed by assembling $(\ref{QuadNonArch})$ and $(\ref{QuadArch})$ into  global inequalities using $(\ref{ExampleIdentity})$, and noting that $h_{\psi_c}(\infty)=0$ since $\psi_c$ fixes $\infty$.
\end{proof}

\subsection{Latt\`es maps associated to a family of elliptic curves}\label{LattesSect}  Let $a$ and $b$ be positive integers, and let $E/\QQ$ be the elliptic curve given by the Weierstrass equation
\begin{equation}\label{WeierstrassEq}
y^2=P(x)=x(x-a)(x+b).
\end{equation}
Let $\psi_E:\PP^1\to\PP^1$ be the map obtained from the action of the doubling map $[2]:E\to E$ on the $x$-coordinate of $E$.  More precisely, we have the following commutative diagram and explicit formula:
\begin{equation}\label{LattesMap}
\begin{CD}
E  @> [2] >>   E \\
@V x VV                                    @VV x V \\
\PP^1           @> \psi_E >>      \PP^1
\end{CD}
\hskip2cm
\psi_E(x) = \frac{(x^2+ab)^2}{4x(x-a)(x+b)}.
\end{equation}
For more details about such maps, which are known as Latt\`es maps, see \cite{SilvermanADS} $\S$~6.4.  The explicit formula in $(\ref{LattesMap})$ can be deduced from the duplication formula III.2.3(d) in \cite{SilvermanI}.

\begin{prop}\label{ExamplePropLattes}
Let $\sigma:\PP^1\to\PP^1$ be the squaring map $\sigma(x)=x^2$ and let $\psi_E:\PP^1\to\PP^1$ be the Latt\`es map $(\ref{LattesMap})$ associated to the elliptic curve $E/\QQ$ defined in $(\ref{WeierstrassEq})$, where $a$ and $b$ are positive integers.  Then there exists an absolute constant $c_1>0$ such that
\begin{equation}\label{LattesPropIneq}
\log\sqrt{ab}\leq \langle\sigma,\psi_E\rangle\leq c_1+\log\sqrt{ab}.
\end{equation}
\end{prop}

\noindent
{\em Remark.}  An explicit value for the constant $c_1$ could in principle be calculated, although we will not attempt to do so.

Before we prove Proposition~\ref{ExamplePropLattes} we will need to have an explicit formula for the canonical measure $\mu_{\psi_E}$ at the archimedean place.  Since this may be useful in other settings, we will state this formula in slightly more generality than we need in the following lemma.

\begin{lem}\label{ArchLattesLem}
Let $E/\CC$ be an elliptic curve given by a Weierstrass equation $y^2=P(x)=x^3+Ax^2+Bx+C$, and let $\psi_E:\PP^1\to\PP^1$ be the Latt\`es map satisfying $x\circ[2]=\psi_E\circ x$, where $[2]:E\to E$ is the doubling map and $x:E\to\PP^1$ is the $x$-coordinate map. Then the canonical measure $\mu_{\psi_E}$ on $\PP^1(\CC)$ is given by
\begin{equation}\label{LattesCanMeasure}
\mu_{\psi_E}(x)=C_P^{-1}|P(x)|^{-1}\ell(x),
\end{equation}
where $\ell(x)$ is the measure on $\PP^1(\CC)=\CC\cup\{\infty\}$ which coincides with Lebesgue measure on $\CC$, and $C_P=\int|P(x)|^{-1}d\ell(x)$.
\end{lem}
\begin{proof}
Let $\mu(x)$ denote the measure on the right-hand-side of $(\ref{LattesCanMeasure})$.  To prove $(\ref{LattesCanMeasure})$ it suffices to show that $\psi_E^*\mu=4\mu$, because this condition, along with the normalization $\int1d\mu(x)=1$ and the property of having no point-masses, characterizes the measure $\mu_{\psi_E}$; see \cite{Ljubich}.

Let $\Lambda$ be the (unique) lattice in $\CC$ such that $g_2(\Lambda)=-4P'(\frac{-A}{3})$ and $g_3(\Lambda)=-4P(\frac{-A}{3})$, where $g_2(\Lambda)$ and $g_3(\Lambda)$ are the usual modular invariants associated to $\Lambda$; see \cite{SilvermanI} Thm.~VI.5.1.  Define functions $X:\CC/\Lambda\to\CC$ and $Y:\CC/\Lambda\to\CC$ by $X(z)=\wp(z)-\frac{A}{3}$ and $Y(z)=\frac{1}{2}\wp'(z)$, where $\wp(z)$ denotes the elliptic Weierstrass function on $\CC/\Lambda$.  Then the map $\CC/\Lambda\to E(\CC)$ given by $z\mapsto(X(z),Y(z))$ is a complex-analytic isomorphism, and $X\circ[2]=\psi_E\circ X$, where $[2]:\CC/\Lambda\to\CC/\Lambda$ is the doubling map $[2]z=2z$.  Since $\deg(X)=2$ and the Jacobian of $z\mapsto X(z)$ is $|X'(z)|^2=|\wp'(z)|^2=4|Y(z)|^2$, we have
\begin{equation*}
\int_{\CC/\Lambda} f(X(z))4|Y(z)|^2dL(z)=2\int_{\PP^1(\CC)} f(x)d\ell(x)
\end{equation*}
whenever $f:\PP^1(\CC)\to\RR$ is $\ell$-integrable; here $L(z)$ denotes Lebesgue measure on the quotient $\CC/\Lambda$.  Taking $f(x)=g(x)/4C_P|P(x)|$ and using the fact that $4|P(x)|=4|y|^2=4|Y(z)|^2$, this becomes
\begin{equation*}
4C_P^{-1}\int_{\CC/\Lambda} g(X(z))dL(z)=2\int_{\PP^1(\CC)} g(x)d\mu(x),
\end{equation*}
which means that $X^*\mu=4C_P^{-1}L$.  Since $L$ is Lebesgue measure on $\CC/\Lambda$ and $\deg([2])=4$, we have $[2]^*L=4L$.  Using this and the fact that $X\circ[2]=\psi_E\circ X$, we have
\begin{equation*}
X^*(4\mu)=16C_P^{-1}L=[2]^*(4C_P^{-1}L)=[2]^*X^*\mu=X^*\psi_E^*\mu,
\end{equation*}
which implies that $4\mu=\psi_E^*\mu$, as desired.
\end{proof}

\begin{proof}[Proof of Proposition~\ref{ExamplePropLattes}]
We are going to show that
\begin{equation}\label{LattesNonArch}
\int\log^+|x|_vd\mu_{\psi_E,v}(x)=0
\end{equation}
at the non-archimedean places $v$ of $\QQ$, and that
\begin{equation}\label{LattesArch}
\log\sqrt{ab}\leq \int\log^+|x|d\mu_{\psi_E}(x)\leq c_1+\log\sqrt{ab}
\end{equation}
at the archimedean place, for some absolute constant $c_1>0$ (since $\QQ$ has just one archimedean place, corresponding to the usual absolute value $|\cdot|$ on $\CC$, we indicate that we are working over this place by omitting the subscript $v$ in the notation).  Combining $(\ref{LattesNonArch})$ and $(\ref{LattesArch})$ via $(\ref{ExampleIdentity})$, and noting that $h_{\psi_E}(\infty)=0$ since $\psi_E$ fixes $\infty$, we deduce the desired inequalities $(\ref{LattesPropIneq})$.

First suppose that $v$ is non-archimedean.  Let $B_v(0,1)=\{x\in\CC_v\mid|x|_v\leq1\}$ denote the closed unit disc in $\CC_v$, and let $\Bsf_v(0,1)=\{x\in\Psf_v^1\mid|x|_v\leq1\}$ denote its closure in $\Psf_v^1$ (the extension to $\Psf_v^1$ of the function $x\mapsto|x|_v$ is described at the end of $\S$~\ref{BerkovichSect}).  Given a point $x\in\CC_v$ such that $|x|_v>1$, the ultrametric inequality and the explicit formula $(\ref{LattesMap})$ imply that $|\psi_E(x)|_v=|x|_v/|4|_v$, since $|a|_v\leq1$ and $|b|_v\leq1$.  Since $\PP^1(\CC_v)\setminus B_v(0,1)$ is dense in $\Psf_v^1\setminus \Bsf_v(0,1)$, the formula $|\psi_E(x)|_v=|x|_v/|4|_v$ holds for all $x\in\Psf_v^1\setminus \Bsf_v(0,1)$ by continuity.  Since $|4|_v\leq1$, we deduce that $|\psi_E(x)|_v>1$ for all $x\in\Psf_v^1$ with $|x|_v>1$; in other words $\psi_E(\Psf_v^1\setminus \Bsf_v(0,1))\subseteq\Psf_v^1\setminus \Bsf_v(0,1)$, which in turn implies that
\begin{equation}\label{UnitDiscPullback}
\psi_E^{-1}(\Bsf_v(0,1))\subseteq\Bsf_v(0,1).
\end{equation}

Recall from $\S$~\ref{CanonicalMeasureSect} that the canonical measure $\mu_{\psi_E,v}$ is defined as a weak limit of measures supported the sets $\psi_E^{-k}(\zeta_{0,1})$ of pullbacks of the point $\zeta_{0,1}\in\Psf_v^1$.  Since $\zeta_{0,1}\in\Bsf_v(0,1)$, it follows from $(\ref{UnitDiscPullback})$ that all of these measures, and thus their limit $\mu_{\psi_E,v}$, are supported on $\Bsf_v(0,1)$.  Since the integrand $\log^+|x|_v$ vanishes on on $\Bsf_v(0,1)$, it vanishes on the support of $\mu_{\psi_E,v}$, whereby the integral in $(\ref{LattesNonArch})$ vanishes as desired.

It now remains only to prove the inequalities $(\ref{LattesArch})$ at the archimedean place.  Given $\alpha\geq0$ and $\beta\geq0$, define
\begin{equation*}
F(\alpha,\beta) =\int_{0}^{2\pi}\frac{1}{|\alpha e^{i\theta}-1||\beta e^{i\theta}+1|}d\theta.
\end{equation*}
Elementary arguments show that $F(\alpha,\beta)$ is nonnegative, finite when $\alpha\neq1$ and $\beta\neq1$, and that it satisfies the symmetry property $F(\alpha,\beta)=F(\beta,\alpha)$ and the functional equation $F(1/\alpha,1/\beta)=\alpha\beta F(\alpha,\beta)$.

Using polar coordinates $x=re^{i\theta}$ we have
\begin{equation}\label{CPCalc}
C_P = \int\frac{1}{|P(x)|}d\ell(x) = \frac{1}{ab}\int_{0}^{+\infty}F(r/a,r/b)\,dr
\end{equation}
and
\begin{equation}\label{LattesIntCalc}
\int\frac{\log^+|x|}{|P(x)|}d\ell(x) = \frac{1}{ab}\int_{1}^{+\infty}F(r/a,r/b)\log r \,dr.
\end{equation}
Note also that
\begin{equation}\label{IntegralVanishes}
\frac{1}{ab}\int_{0}^{+\infty}F(r/a,r/b)\log(r/\sqrt{ab}) \,dr=0.
\end{equation}
To see this, denote by $I$ the left-hand-side of $(\ref{IntegralVanishes})$.  Using the equation $F(1/\alpha,1/\beta)=\alpha\beta F(\alpha,\beta)$ followed by the change of coordinates $r\mapsto ab/r$ it follows that $I=-I$, whereby $I=0$ as desired.  Combining $(\ref{CPCalc})$, $(\ref{LattesIntCalc})$, and $(\ref{IntegralVanishes})$, and dividing through by $C_P$, we deduce that
\begin{equation}\label{ThetaIdentity}
\int \log^+|x| d\mu_{\psi_E}(x)=\frac{1}{C_P}\int\frac{\log^+|x|}{|P(x)|}d\ell(x) =\Theta_{a,b}+\log\sqrt{ab},
\end{equation}
where
\begin{equation}\label{Theta}
\Theta_{a,b}=\frac{1}{abC_P}\int_{0}^{1}F(r/a,r/b)\log(1/r) \,dr.
\end{equation}
Obviously $\Theta_{a,b}\geq0$ since the integrand in $(\ref{Theta})$ is nonnegative, which along with $(\ref{ThetaIdentity})$ implies the lower bound in $(\ref{LattesArch})$.  In order to prove the upper bound in $(\ref{LattesArch})$, it suffices to show that
\begin{equation}\label{ThetaLimsup}
\limsup_{\max\{a,b\}\to+\infty}\Theta_{a,b}<+\infty.
\end{equation}
We will omit the proof of $(\ref{ThetaLimsup})$, which is lengthy but straightforward; it uses only trivial upper and lower bounds on the function $F(\alpha,\beta)$ along with some elementary calculus.
\end{proof}

\subsection{The height difference bound revisited}  Let $\sigma_\alpha(x)=\alpha-(\alpha-x)^2$, $\psi_c(x)=x^2+c$, and $\psi_E(x)=(x^2+ab)^2/4x(x-a)(x+b)$ be the examples considered in $\S$~\ref{COCSect}, $\S$~\ref{QuadPolySect}, and $\S$~\ref{LattesSect}, respectively.  Combining Theorem~\ref{HeightDiffProp} with Propositions~\ref{ExamplePropCOC}, \ref{ExamplePropQuad}, and \ref{ExamplePropLattes}, we have
\begin{equation}\label{ExampleInequalities}
\begin{split}
h_{\sigma_\alpha}(x)-h_\st(x) & \leq h_\st(\alpha) + c_2 \\
h_{\psi_c}(x)-h_\st(x) & \leq (1/2)h_\st(c) + c_3 \\
h_{\psi_E}(x)-h_\st(x) & \leq \log\sqrt{ab} + c_4,
\end{split}
\end{equation}
for all $x\in\PP^1(\Kbar)$, where $c_2=\frac{3\sqrt{3}}{4\pi}L(2,\chi)+\log2$, $c_3=\log 4$, and $c_4=c_1+\log2$, where $c_1$ is the absolute constant appearing in Proposition~\ref{ExamplePropLattes}.


\medskip


\medskip

\begin{thebibliography}{1}

\bibitem{Arakelov}
S. Arakelov,
\newblock `Intersection theory of divisors on an arithmetic surface',
\newblock {\em Math. USSR Izvestija} 8 (1974) 1167-1180.

\bibitem{BakerDeMarco}
M. Baker and L. DeMarco,
\newblock `Preperiodic points and unlikely intersections',
\newblock Preprint, http://arxiv.org/abs/0911.0918

\bibitem{BakerRumely}
M. Baker and R. Rumely,
\newblock `Equidistribution of small points, rational dynamics, and potential theory',
\newblock {\em Ann. Inst. Fourier (Grenoble)} 56 no. 3 (2006) 625--688.

\bibitem{BakerRumelyBook}
M. Baker and R. Rumely.
\newblock {\em Potential Theory on the Berkovich Projective Line}.
\newblock AMS Surveys and Mathematical Monographs series, to appear.

\bibitem{BenedettoBriendPerdry}
R. Benedetto, J.-Y. Briend, and H. Perdry,
\newblock `Dynamique des polynomes quadratiques sur les corps locaux',
\newblock {\em Journal de Th\'eorie des Nombres de Bordeaux} 19 (2007), 325-336.

\bibitem{Berkovich}
V. Berkovich,
\newblock {\em Spectral Theory and Analytic Geometry over Non-Archimedean Fields},
\newblock AMS Mathematical Surveys and Monographs 33 (AMS, Providence, 1990).

\bibitem{Bilu}
Y. Bilu.
\newblock Limit distribution of small points on algebraic tori.
\newblock {\em Duke Math.~J.}, 89: 465--476, 1997.

\bibitem{BombieriGubler}
E. Bombieri and W. Gubler,
\newblock {\em Heights in Diophantine Geometry},
\newblock Cambridge University Press, New York, 2006.

\bibitem{BostGilletSoule}
J.-B. Bost, H. Gillet, and C. Soul\'e,
\newblock `Heights of projective varieties and positive green forms'
\newblock {\em J. Amer. Math. Soc.} 7 (1994), 903-1027.

\bibitem{Boyd}
D. Boyd,
\newblock `Speculations concerning the range of Mahler's measure',
\newblock {\em Canadian Mathematical Bulletin} 24 (1981), 453-469.

\bibitem{Brolin}
H. Brolin,
\newblock `Invariant sets under iteration of rational functions',
\newblock {\em Ark. Mat.} 6 (1965) 103-144.

\bibitem{CallSilverman}
G. S. Call and J. H. Silverman,
\newblock `Canonical heights on varieties with morphisms',
\newblock {\em Compositio Mathematica} 89 no. 2 (1993) 163-205.

\bibitem{ChambertLoir}
A. Chambert-Loir,
\newblock `Mesures et \'equidistribution sur les espaces de Berkovich',
\newblock {\em J. f\"ur die reine und angewandte Mathematik} 595 (2006) 215-235.

\bibitem{ChambertLoirThuillier}
A. Chambert-Loir and A. Thuillier,
\newblock `Mesures de Mahler et \'equidistribution logarithmique',
\newblock {\em Annales de l'Institut Fourier} 59 (2009) 977-1014

\bibitem{Deligne}
P. Deligne,
\newblock `Le d\'eterminant de la cohomologie',
\newblock Current trends in arithmetical algebraic geometry, Contemporary Mathematics, vol. 67, American Mathematical Society, Providence, 1987, pp. 94-177.

\bibitem{Faltings}
G. Faltings,
\newblock `Calculus on arithmetic surfaces',
\newblock {\em Ann. of Math.} (2) 119 (1984), no. 2, 387-424.

\bibitem{FavreRiveraLetelier}
C. Favre and J. Rivera-Letelier,
\newblock `{\'E}quidistribution quantitative des points de petite hauteur sur la droite projective'
\newblock {\em Math. Ann.} (2) 335 (2006) 311Ð361.

\bibitem{FavreRiveraLetelier2}
C. Favre and J. Rivera-Letelier,
\newblock `Th\'eorie ergodique des fractions rationnelles sur un corps ultram\'etrique'
\newblock Preprint (2007), arxiv.org/abs/0709.0092.

\bibitem{FreireLopesMane}
A. Freire, A. Lopes, and R. Ma\~n\'e,
\newblock `An invariant measure for rational maps',
\newblock {\em Bol. Soc. Brasil. Mat.} 14 (1983) 45-62.

\bibitem{KawaguchiSilverman}
S. Kawaguchi and J. Silverman,
\newblock Dynamics of projective morphisms having identical canonical heights.
\newblock {\em Proc. London Math. Soc.} 95 (2007) 519-544.

\bibitem{Ljubich}
M. Ljubich.
\newblock Entropy properties of rational endomorphisms of the Riemann sphere.
\newblock {\em Ergodic Theory Dynamical Systems} 3 (1983) 351-385.

\bibitem{Milnor}
J. Milnor.
\newblock Dynamics in One Complex Variable.
\newblock Annals of Mathematics Studies, vol. 160, Princeton University Press, Princeton, 2006.

\bibitem{Mimar}
A. Mimar.
\newblock On the preperiodic points of an endomorphism of $\PP^1\times\PP^1$ which lie on a curve.
\newblock Ph.D. Dissertation, Columbia University (1997).

\bibitem{SilvermanI}
Joseph~H. Silverman.
\newblock {\em The Arithmetic of Elliptic Curves}, volume 106 of {\em Graduate
 Texts in Mathematics}.
\newblock Springer-Verlag, New York, 1986.

\bibitem{SilvermanII}
J. H. Silverman.
\newblock {\em Advanced Topics in the Arithmetic of Elliptic Curves}, volume
  151 of {\em Graduate Texts in Mathematics}.
\newblock Springer-Verlag, New York, 1994.

\bibitem{SilvermanADS}
J. H. Silverman.
\newblock \emph{The Arithmetic of Dynamical Systems}, volume 241 of {\em Graduate Texts in Mathematics}.
\newblock Springer, New York, 2007.

\bibitem{SzpiroTucker}
L. Szpiro and T.J. Tucker.
\newblock Equidistribution and generalized Mahler mesures,
\newblock Preprint (2009).

\bibitem{SzpiroUllmoZhang}
L. Szpiro, E. Ullmo, and S. Zhang.
\newblock {\'E}quir{\'e}partition des petits points.
\newblock {\em Invent. Math.}, 127:337-347, 1997.

\bibitem{Thuillier}
A. Thuillier.
\newblock Th{\'e}orie du potentiel dur les courbes en g{\'e}om{\'e}trie analytique non archim{\'e}dienne.
Applications {\`a} la th{\'e}orie d'Arakelov.
\newblock Ph.D. thesis, University of Rennes, 2005.

\bibitem{vanderWaerden}
B. van der Waerden.
\newblock {\em Modern Algebra. Vol. II}.
\newblock Frederick Ungar Publishing Co., New York, 1949.

\bibitem{Zhang}
S. Zhang.
\newblock Small points and adelic metrics,
\newblock {\em J. Alg. Geom.} 4(2) (1995) 281-300.

\end{thebibliography}
\end{document}